\documentclass[a4paper,12pt]{article} %

\usepackage{mathtools}
\usepackage{amsmath} %
\usepackage{amsthm} %
\usepackage{amssymb} %

\usepackage{epsfig} %
\usepackage{psfrag} %

\usepackage[mathscr,mathcal]{eucal} %
\usepackage{enumerate} %

\usepackage{dsfont}

\newtheorem  {theorem}       {Theorem}
\newtheorem  {lemma} [theorem]        {Lemma}
\newtheorem  {corollary} [theorem]    {Corollary}
\newtheorem  {proposition}[theorem]   {Proposition}

\newtheorem* {theorem*}      {Theorem}
\newtheorem* {lemma*}        {Lemma}
\newtheorem* {corollary*}    {Corollary}
\newtheorem* {proposition*}  {Proposition}
\newtheorem* {definition*}   {Definition}
\newtheorem* {remark*}       {Remark}

\newtheorem  {remark}        {Remark}[section]
\newtheorem* {remarks*}      {Remarks}

\newtheorem* {claim*}        {Claim}

\newcommand{\G}{\ensuremath{\mathcal{G}}}

\title
{Well-posedness of the mean field forest fire age evolution equation}

\author {Edward Crane\footnote{University of Bristol and Heilbronn Institute for Mathematical Research}}
\date{July 11, 2020}

\begin{document}
\maketitle

\begin{abstract}
We prove the well-posedness of a differential equation that describes the evolution of the large-system limit of the empirical age measure in the mean field forest fire model of R\'ath and T\'oth \cite{RathToth}. This forest fire model is a random graph process on $n$ vertices, whose dynamics combine the Erd\H{o}s-R\'enyi dynamics with a Poisson rain of lightning strikes. All edges in any connected component are deleted as soon as any of its vertices is struck by lightning. Each vertex has an age, which increases at rate $1$ but is reset to $0$ each time it burns. We consider the asymptotic lightning regime in which the model displays self-organized criticality.   Crane, R\'ath and Yeo \cite{CraneRathYeo} take the initial state to be an inhomogeneous random graph whose edge probabilities depend on the ages of the vertices. They show that as $n \to \infty$ the empirical age distribution converges as a process to the solution of a  deterministic autonomous differential equation. It is a nonlinear age-dependent population dynamics model whose age-specific mortality modulus involves the leading eigenfunction of the branching operator of an associated multitype branching process. The differential equation displays self-organized criticality in the sense that the leading eigenvalue of the branching operator is held at $1$ without this being imposed as a boundary condition.
\end{abstract}


\section{Introduction}

 We prove the well-posedness of a nonlinear age-dependent population dynamics model that describes the evolution over time of the limiting distribution of vertex ages in the mean field forest fire model of R\'ath and Toth \cite{RathToth}. This age evolution equation is described below in \S\ref{SS: problem statement}. The mean field forest fire model is a dynamic random graph model that is rigorously shown to exhibit self-organized criticality in its hydrodynamic limit. The population dynamics model is deterministic and displays self-organized criticality. 
 
 We describe the mean field forest fire model below in section~\S\ref{SS: MFFF}. In \cite{CraneRathYeo} we decorate the mean field forest fire model with vertex ages, and we take the initial graph state to be random, distributed as an inhomogeneous random graph whose edge probabilities depend in a certain way on the initial ages of the vertices. This class of of random graphs is natural because it is preserved under the forest fire dynamics. Under these assumptions we show that there is a limiting empirical distribution of vertex ages at all times, as the model size tends to infinity. We also derive in \cite{CraneRathYeo} a system of differential equations satisfied by this limiting age distribution. The purpose of the present paper is to prove that this autonomous system is well-posed. 

 The author would like to thank Bal\'azs R\'ath for his assistance.

\subsection{Problem statement and main results}\label{SS: problem statement} 
Denote by $\mathcal{P}([0,\infty))$ the space of Borel probability measures on $[0,\infty)$, by $\mathcal{M}_{1}^{+}([0,\infty))$ the space
 of positive Borel measures on $[0,\infty)$ with finite first moment, and by $\mathcal{P}_{1}^{+}([0,\infty))$ their intersection.
  The Wasserstein $W_1$ distance $W_1(\mu,\nu)$ between two measures $\mu,\nu \in \mathcal{P}([0,\infty))$ is the $L^1$ distance between their cumulative distribution functions, which could be infinite. However, $W_1$ restricts to a complete metric on the subspace $\mathcal{P}_1([0,\infty))$. It is also called the earth-mover's distance or Kantorovich-Rubinstein metric. We discuss $W_1$ in detail in \S\ref{SS: W_1 metric}.

For $\pi \in \mathcal{P}([0,\infty))$ we define a linear integral operator $\mathcal{L}_\pi$ on the Hilbert space $L^2(\pi)$ by
\begin{equation}\label{eq: operator definition} \mathcal{L}_{\pi}f (s) = \int (x \wedge s) f(x)\,d\pi(x)\,.\end{equation}
Here $x \wedge s$ denotes the minimum of $x$ and $s$.

Suppose $\pi \neq \delta_0$ (where $\delta_0$ means the atom of unit mass at $0$). It is shown in \cite{CraneRathYeo} that $\mathcal{L}_\pi$ is a nontrivial positive semidefinite self-adjoint Hilbert-Schmidt operator. In the context of forest fires, $\mathcal{L}_\pi$ is the branching operator of a multitype Galton-Watson process that approximates the local structure of the forest fire graph. We remark that the eigenbasis of $\mathcal{L}_\pi$ also arises in the Kosambi-Karhunen-Lo\`eve expansion of standard Brownian motion (considered as an element of $L^2(\pi)$) as a random series of orthogonal functions with independent Gaussian coefficients, which explains as much of the variance as possible in each finite initial subseries. 

$\mathcal{L}_\pi$ has a simple leading or principal eigenvalue $\lambda > 0$, that is equal to its spectral radius and its operator norm on $L^2(\pi)$. 
$\mathcal{L}_\pi$ has a unique eigenfunction $\theta \in L^2(\pi)$ with eigenvalue $\lambda$, subject to the normalization $$\langle \theta, \mathbf{1}\rangle_\pi = \int \theta(x) \,d\pi(x)  = 1\,.$$ We will abuse notation by letting $\theta$ also denote the continuous increasing function on $[0,\infty)$ defined for $s \in [0,\infty)$ by
\begin{equation}\label{eq: theta e'vec} \lambda \theta(s) = \int (x \wedge s) \theta(x) \,d\pi(x)\,.\end{equation}
 The probability measure $\mu$ defined by $\frac{d\mu}{d\pi}(x) = \theta(x)$ will play an important role. We show in \cite{CraneRathYeo} that $\theta$ is bounded, using the connection to multitype branching processes. Hence $\mu \in \mathcal{P}_1([0,\infty))$.  In this paper we will give a different and self-contained proof that $\theta$ is bounded, which has the advantages of yielding an explicit bound, and shows that $\theta$ is locally uniformly bounded on $\mathcal{P}_1([0,\infty))\setminus\{\delta_0\}$.
 
 We say that $\pi$ is \emph{age-subcritical} if $\lambda < 1$, \emph{age-critical} if $\lambda = 1$ and \emph{age-supercritical} if $\lambda > 1$.  It is shown in \cite{CraneRathYeo} that these conditions correspond to sub-criticality, criticality and supercriticality of the multitype Galton-Watson process whose branching operator is $\mathcal{L}_\pi$.



Our main result is the following well-posedness theorem.
\begin{theorem}\label{T: well-posedness theorem}
 Consider the initial value problem
 \begin{equation}\label{eq: IVP} \begin{rcases} \frac{d}{dt}\pi_t = -\delta_0' \ast \pi_t - \varphi(t)(\mu_t - \delta_0)\\
\displaystyle{\frac{d\mu_t}{d\pi_t}   = \theta_t}\,,\quad \int \theta_t(s) \,d\pi_t(s) = 1\\
\mathcal{L}_{\pi_t} \theta_t  = \lambda_t \theta_t\,,\quad\text{and}\quad \lambda_t = \left\|\mathcal{L}_{\pi_t}\right\|_{L^2\left(\pi_t\right)}\,.\end{rcases}\end{equation} 
Here $t$ ranges over $[0,T]$, $\pi_t$ is understood to take values in $\mathcal{P}([0,\infty))$, and $\varphi: [0,T] \to [0,\infty)$ is understood to be a continuous control function. The meaning of the differential equation is that for every compactly supported and continuously differentiable test function $f: \mathbb{R} \to \mathbb{R}$ we have
\begin{equation}\label{eq: IVP meaning}
\frac{d}{dt} \int\! f(s)\,\mathrm{d}\pi_t(s) =
 \int\! f'(s)\,\mathrm{d}\pi_t(s) - \int\! f(s)\varphi(t)\theta_t(s)\,\mathrm{d}\pi_t(s) + \varphi(t) f(0)\,.
\end{equation}
Let $\pi_0 \in \mathcal{P}_1([0,\infty))$ be age-critical and let $T >0$. Then the system~\eqref{eq: IVP} together with the equation \begin{equation} \label{eq: phi defined by theta}\varphi(t) = \Phi(\pi_t) := \left(\int \theta_t(s)^3 \,d\pi_t(s)\right)^{-1} \end{equation} has a unique solution over $[0,T]$. It satisfies $\lambda_t = 1$ and $\pi_t \in \mathcal{P}_1([0,\infty))$ for all $t$. Over $[0,T]$ the solution $\pi_t$ moves Lipschitz continuously in the Wasserstein $W_1$ metric, and the normalized eigenfunction $\theta_t$ moves Lipschitz continuously in $L^1([0,\infty))$. The dependence of each of $\pi_t, \theta_t$, and $\varphi(t)$ on $\pi_0$ is locally Lipschitz, uniformly over $[0,T]$.
\end{theorem}
The proof of Theorem~\ref{T: well-posedness theorem} occupies Section~\ref{S: main proof}. The existence statement in Theorem~\ref{T: well-posedness theorem} is already implied by the limit theorems in \cite{CraneRathYeo} that we will describe in~\S\ref{SS: MFFF}. The uniqueness and continuous dependence on initial conditions are proved using a Gr\"onwall argument, which relies on proving the local Lipschitz dependence of $\lambda$ and $\theta$ on $\pi$.

\subsection{Background: the mean field forest fire model}\label{SS: MFFF}
 The mean field forest fire model (MFFF) is defined by modifying the dynamical Erd\H{o}s-R\'enyi random graph process on $n$ vertices by the addition of a Poisson rain of lightning strikes, striking each vertex independently at rate $\lambda(n)$, where $\lambda(n) \to 0$ but $n\lambda(n) \to \infty$ as $n \to \infty$. When lightning strikes a vertex, all edges in the connected component (cluster) of that vertex are instantaneously deleted, but the vertices survive and continue to form new edges as in the Erd\H{o}s-R\'enyi dynamics. Then \cite{RathToth} shows that the model displays \emph{self-organized criticality} in the limit $n \to \infty$.
  This means that its stochastic dynamics drive it into a state where the cluster size distribution exhibits polynomial decay, after which it remains in such a critical state. However, a giant component never forms, which is to say that the model never becomes supercritical. 
  
 The cluster size distribution $\left(v_k^n(t)\right)_{k=1}^\infty$ expresses the proportion of vertices in the model at time $t$ that belong to clusters of size $k$. The sequence $\left(v_k^n(t)\right)_{k=1}^\infty$ is a Markov process on its own, a modification of the Marcus-Lushnikov coagulation process with multiplicative kernel. Norris \cite{Norris} showed that the Marcus-Lushnikov process with multiplicative kernel has the modified Smoluchowski (or Flory) equation as its hydrodynamic limit. R\'ath and T\'oth \cite{RathToth} considered a sequence of MFFF processes where $n \to \infty$ and $v_k^n(0) \to v_k(0)$ as $n \to \infty$ for each $k \ge 1$. They showed that $v_k^n(t) \to v_k(t)$ in probability, where the limit $v_k(t)$ is deterministic, and the vector $\left(v_k(t)\right)_{k=1}^\infty$ satisfies the \emph{critical forest fire equations}. This is a coupled system of infinitely many ODEs, similar to the Flory equations with multiplicative kernel:
 $$ \frac{d}{dt} v_k(t) = \mathbf{1}_{(k=1)} \varphi(t) - k v_k(t) + \sum_{i + j = k} i v_i(t)v_j(t) \quad \text{ for each $k \ge 1$,} $$ 
 $$ \sum_{k=1}^\infty v_k(t) = 1\,.$$
By means of the Laplace transform the critical forest fire equations may be expressed as a controlled inviscid Burgers equation. A crucial role is played by the \emph{control function} $\varphi(t)$. This function represents the limiting rate (per vertex) at which vertices are burned and hence is also the limiting birth rate of singleton vertices, in the limit $n \to \infty$. $\varphi(t)$ is only determined implicitly by the critical forest fire equations. R\'ath and T\'oth assume that $\sum_{k=1}^\infty k^3 v_k(0) < \infty$, which implies the initial configuration is subcritical. They prove that the critical forest fire equations have a unique solution, and describe it as follows.  There is a \emph{gelation time} $t_{gel}  = \left(\sum_{k=1}^\infty k v_k(0)\right)^{-1}$, such that $\varphi(t) = 0$ for $t < t_{gel}$, but $\varphi(\cdot)$ is strictly positive and continuous on $[t_{gel}, \infty)$.
Moreover, for $t \ge t_{gel}$, 
\begin{equation}\label{eq: phi defined by tail} \sqrt{\frac{2\varphi(t)}{\pi}} = \lim_{m \to \infty} m^{1/2} \sum_{k=m}^\infty v_k(t)\,.\end{equation}
 In \cite{CraneRathYeo} it is noted that the technical results of \cite{RathToth} may be extended to cover some cases where the limiting initial configuration is critical, subject to an analytic condition on the behaviour near $z = 1$ of the probability generating function $ \sum_{k=1}^\infty v_k(0) z^k$. In this case $t_{gel} = 0$.

\cite{CraneFreemanToth} studies the stochastic process of the cluster size of a tagged vertex in the MFFF model, proving that in the limit $n \to \infty$ this tagged cluster size becomes Markovian on its own, with a deterministic time-dependent generator that is given in terms of the solution of the critical forest fire equations. In other words, the system exhibits \emph{propagation of chaos}. The limiting probability that a vertex survives unburned from time $t_1$ until a later time $t_2$, either unconditionally or conditioned on its cluster size at time $t_1$, is expressed in terms of characteristic curves of the controlled Burgers equation.

\subsection{Limiting age evolution in the mean field forest fire model} 

The limiting age evolution equations \eqref{eq: IVP} + 
\eqref{eq: phi defined by theta} are derived in \cite{CraneRathYeo} for the \emph{mean field forest fire with ages}, (MFFFA).

 To explain the model, we first define an \emph{age-driven inhomogeneous random graph} on the vertex set $\{1, \dots, n\}$, where each vertex $v$ is labelled with an age $a(v)$. The ages could be deterministic or random. Conditional on the ages, for each pair of distinct vertices $v,w$, independently, an edge joins $v$ and $w$ with probability $1 - \exp(-(a(v) \wedge a(w))/n)$. This is a special case of the inhomogeneous random graph (IRG) studied by Bollob\'as, Janson and Riordan~\cite{BollobasJansonRiordan}.

The model $\textup{MFFFA}(n, \underline{a}_0^n, \lambda)$ is a stochastic process $\mathcal{G}_t^n$ taking values in the set of graphs on the vertex set $\{1, \dots, n\}$ with vertices labelled by ages: at time $t \in [0,\infty)$ each vertex $v$ has age $a_t^n(v) \in [0,\infty)$. The vector of $n$ time-dependent ages is denoted $\underline{a}_t^n$. The initial age vector $\underline{a}_0^n$ may be random. The initial graph $\mathcal{G}_0^n$ is a sample of the age-driven inhomogenous random graph with ages $\underline{a}_0^n$. The dynamics of the model are as follows. The age of each vertex increases at rate 1. Lightning strikes each vertex independently at rate $\lambda$. When a cluster is struck by lightning, all of its edges are immediately deleted and the age of each of its vertices is immediately reset to $0$.  

The dynamics of the MFFFA model preserves the class of mixtures of age-driven inhomogeneous random graphs. More precisely, for any fixed time $t$, if we condition the model on the vector $\underline{a}_s^n$ of vertex ages over all times $s \in[0,t]$, then the graph state of the model at time $t$ is conditionally distributed as the age-driven inhomogeneous random graph driven by the ages $\underline{a}_t^n$. This property is used in \cite{CraneRathYeo} to show that $\underline{a}_t^n$ is a Markov process on its own. 

We do not insist that the initial ages are independent, since the independence of the vertex ages is not preserved by the forest fire dynamics. However, exchangeability is preserved by the dynamics, so it is quite natural to let the initial ages be exchangeable. In both \cite{RathToth} and \cite{CraneRathYeo} the most important example is the \emph{monodisperse} initial condition, in which there are no edges at time $0$ and all $n$ vertices initially have age $0$. 
 
The first main result of \cite{CraneRathYeo} gives conditions under which the empirical age distribution at each time $t > 0$ converges in probability (with respect to the topology of weak convergence) to a deterministic probability measure $\pi_t$, as the size of the model tends to infinity. 
\begin{theorem}[{\cite{CraneRathYeo}}]\label{thm_convergence}
Let $(\G^n_t,\,t\in[0,t_{\max}])$ be a family of $\mathrm{MFFFA}(n,\underline{a}^n_0,\lambda(n))$ processes, with lightning rate satisfying $\lambda(n) \to 0$ and $n\lambda(n) \to \infty$ as $n \to \infty$.
Suppose that the initial (random) empirical age measures satisfy $\pi^n_0 \stackrel{\mathbb{P}}{\Rightarrow} \pi_0$, where $\pi_0$ is a deterministic probability measure on $[0,\infty)$ with finite mean that is either age-critical or age-subcritical. Then there exists a deterministic
family of probability measures $(\pi_t)_{0 \leq t \leq t_{\max}}$ on $[0,\infty)$, depending continuously on $t$ with respect to the topology of weak convergence, such that
$ \pi^n_t  \stackrel{\mathbb{P}}{\Rightarrow} \pi_t$ as $n \to \infty$,
where the convergence in probability is with respect to the topology of weak convergence. There is a deterministic finite gelation time $t_{gel}$ which is $0$ if $\pi_0$ is age-critical, but positive if $\pi_0$ is age-subcritical. For $0 \le t < t_{gel}$, $\pi_t$ satisfies the transport equation 
$$\frac{d}{dt} \pi_t = -\delta_0' \ast \pi_t\,,$$ and  $\pi_t$ remains age-subcritical. For $t \ge t_{gel}$, $\pi_t$ is age-critical and satisfies the system \eqref{eq: IVP} \textup{+} \eqref{eq: phi defined by theta}.
\end{theorem}

In the case where the initial age distribution $\pi_0$ is age-subcritical with finite mean, the limiting cluster size distribution exists and has an exponential tail, so it can be used as the limiting initial cluster size distribution $\left(v_k(0)\right)_{k=1}^\infty$ in the main theorem of R\'ath and T\'oth \cite{RathToth}. That theorem determines $t_{gel}$ to be equal to $\left(\sum_{k=1}^\infty kv_k(0)\right)^{-1}$. The function $\varphi:[t_{\mathrm{gel}},\infty) \to (0,\infty)$ defined by equation \eqref{eq: phi defined by theta} coincides with the \emph{control function} $\varphi$ in the statement of R\'ath and T\'oth's main theorem. 

To explain why the age evolution equation takes the form \eqref{eq: IVP} + \eqref{eq: phi defined by theta}, we give a brief sketch of the method of proof in \cite{CraneRathYeo}. The central observation is that conditioned on the ages of the vertices at time $t$, the MFFFA graph is an age-driven inhomogeneous random graph.  The local structure of this graph is well-approximated by a multitype Poisson branching process, whose branching operator is $\mathcal{L}_{\pi_t}$. In this branching process, the offspring of an individual of age (type) $x$ are an almost surely finite set of individuals, whose ages are the points of a Poisson point process with intensity $(x \wedge y)\,d\pi_t(y)$. Because $\theta_t$ is the principal eigenfunction of the branching operator, the tilted measure $\mu_t$ approximates the distribution of ages in very large clusters. Very large clusters account for nearly all of the burning vertices after the gelation time. The quantity $\varphi(t)$ is the limiting total rate of burning (per vertex), so the term $-\varphi(t)\mu_t$ approximates the rate of change in the empirical age distribution due to the removal of burning vertices. The term $\varphi(t)\delta_0$ corresponds to the fact that all vertices burned at time $t$ survive but have their age reset to zero. 

It is shown in \cite{CraneRathYeo} by careful analysis of the singularity at $1$ of the generating function for the total progeny of the multitype branching process, and comparison with equation~\eqref{eq: phi defined by tail}, that the control function $\varphi(\cdot)$ that describes the limiting burning rate in the MFFFA model is described by equation~\eqref{eq: phi defined by theta}. A different (though heuristic) method to obtain~\eqref{eq: phi defined by theta} is to analyze how quickly the giant component of $\mathcal{G}_t$ would begin to grow if the lightning process were switched off at time $t$. (See \cite[Theorem 3.17]{BollobasJansonRiordan} which determines this growth rate.) 

In \cite{CraneRathYeo} it is not shown that the autonomous system describing the age evolution after $t_{gel}$ is well-posed.  However, the results summarized in Theorem~\ref{thm_convergence} do already prove the existence of a solution to $\eqref{eq: IVP} + \eqref{eq: phi defined by theta}$ satisfying $\lambda_t=1$, assuming $\pi_0$ is age-critical with finite mean. 

\subsection{Relation to standard population dynamics models}

 The age-evolution equation is superficially related to the well-known Lotka--Sharpe--McKendrick demographic model, a linear partial differential equation that describes the limiting age distribution for a population in which individuals age at rate 1, die at an age-dependent rate, and reproduce at an age-dependent rate. Like the Lotka-Sharpe-McKendrick model, our age-evolution equation has a birth term, a death term and a transport term. However, it is nonlinear. The coefficients depend in a nonlinear way on the age distribution, but not solely through the total population size as in the nonlinear Gurtin--MacCamy demographic model. In a general age-dependent population dynamics model, the \emph{age-specific mortality modulus} $m(\cdot)$ is the function that specifies the death rate $m(a)$ for individuals of age $a$. In our equation, $m(a) = \varphi(t)\theta_t(a)$.  On its own, $\theta_t$ is a nonlinear function of $\pi_t$. Equation~\eqref{eq: phi defined by theta} together with the normalization $\int \theta_t(x) \,d\pi_t(x) = 1$ causes $m(a)$ to be normalized in the following simple but nonlinear fashion:
  $$ \int m(a)^3\,d\pi(a) = \left(\int m(a)\,d\pi(a)\right)^2\,.$$
The birth rate $\varphi(t)$ equals the total death rate, so that the population size stays fixed, i.e. $\pi$ remains a probability measure.  Although the system is autonomous, involving no boundary conditions, we show that its solutions nevertheless satisfy the boundary condition that the operator $\mathcal{L}_\pi$ remains critical, i.e. $\lambda_t = 1$ for all $t \in [0,T]$. In this sense the system \eqref{eq: IVP} + \eqref{eq: phi defined by theta} exhibits self-organized criticality.

 In the monograph of Webb \cite{Webb}, well-posedness is proven for a very general class of nonlinear age-dependent population dynamics models, under the expected technical assumptions that the age-dependent net emigration (or mortality) rate and the birth rate are locally Lipschitz mappings from $L^1([0,\infty))$ to $L^1(0,\infty))$ and to $[0,\infty)$ respectively. However, the differential equation that we study in this paper is not covered by that general theorem, since the solutions that we must consider take values in the space of non-negative Borel probability measures on $[0,\infty)$ that have finite first moment. The general equation considered in \cite{Webb} deals only with age distributions that are absolutely continuous with respect to Lebesgue measure, so it is phrased in terms of densities taking values in $L^1([0,\infty))$.  Part of the novelty of the present paper is that we work in the larger space of non-negative Borel probability measures. 
 
\subsection{Choice of topology}
 Since the age evolution equation that we consider conserves total population size and preserves the finiteness of the mean age, it is natural to work with a probability metric defined on the space $\mathcal{P}_1([0,\infty))$ of Borel probability measures on $[0,\infty)$ that have finite mean. The $L^1$ norm used in \cite{Webb} would correspond to the total variation distance, but that metric is not appropriate for our problem, since the translation group does not act continuously on the space of Borel probability measures on $\mathbb{R}$ with respect to total variation distance. (It does act continuously on $L^1(\mathbb{R})$.)  We choose instead to use the Wasserstein $W_1$-distance because it is convenient to work with, the solutions of our system move at finite speed with respect to $W_1$, (as we shall prove in Corollary~\ref{Cor: speed limit}), and the continuous dependence of solutions on initial conditions with respect to $W_1$ is meaningful in terms of the forest fire model. It expresses the property of the MFFFA that if we perturb the initial age of each vertex by an amount that is small on average over all the vertices, (and resample the initial graph), then with high probability this causes only small changes in the empirical distribution of ages at later times. 
 
 For a sequence $\pi_n \in \mathcal{P}_1([0,\infty))$,  we have $W_1(\pi_n, \pi) \to 0$ as $n \to \infty$ if and only if $\int x\,d\pi_n(x) \to \int x\,d\pi(x)$ and $\pi_n \to \pi$ in the topology of weak convergence (which means that for every bounded continuous $f: [0,\infty) \to \mathbb{R}$ we have $\int f \,d\pi_n \to \int f \,d\pi$). Thus $\mathcal{P}_1([0,\infty))$ is a closed subspace of $\mathcal{P}([0,\infty))$ with the $W_1$-topology.
 
Note that Theorem~\ref{T: well-posedness theorem} does not state that $\pi_t$ depends continuously on $\pi_0$ with respect to the topology of weak convergence on $\mathcal{P}_1([0,\infty))$. We do not know whether this is true; it does not follow from Theorem~\ref{T: well-posedness theorem}. We will see by an example below that even after restricting to the age-critical subset of $\mathcal{P}_1([0,\infty))$, the first moment $\int x \,d\pi(x)$ is not a continuous functional of $\pi$ with respect to the weak topology. In fact $\mathcal{P}_1([0,\infty))$ is not a closed subspace of $\mathcal{P}([0,\infty))$ with the weak topology.

An apparent obstacle to adapting our proof of Theorem~\ref{T: well-posedness theorem} to work with a metric that metrizes the weak topology, instead of $W_1$, is that the leading eigenvalue of $\mathcal{L}_\pi$ is not a continuous function with respect to the topology of weak convergence on $\mathcal{P}_1([0,\infty))$. Consider the measure $\pi^p := (1-p) \delta_0 + p \delta_{1/p}$, for $p \in (0,1]$. This measure has first moment $1$ and it is age-critical, with $\theta(x) = x \wedge (1/p)$. As $p \to 0$, $\pi^p$ converges weakly to $\delta_0$, which is age-subcritical with $\lambda = 0$. Taking $\pi_0 = \pi^p$ we obtain $\varphi(0) = \Phi(\pi^p) = p^2$ and the initial age-specific mortality rate of the individuals of age $1/p$ is $p$. 
However, $\pi$ does not converge in $W_1$ as $p \to \infty$. In fact $W_1(\pi^p, \pi^q) = 2\left( 1- \frac{p \wedge q}{p \vee q}\right)$,
 demonstrating that $\mathcal{P}_1([0,\infty))$ is not compact in the $W_1$-topology. 
When $p$ is very close to $0$, the solution of \eqref{eq: IVP} + \eqref{eq: phi defined by theta} with initial condition $\pi_0 = \pi^p$ has $\varphi(t)$ very close to $0$ until just before time $1$, so that $\pi_t$ is well approximated by the simple transport equation almost up to time $1$. Around time $1$, $\varphi(t)$ increases quickly to be close to $1$. In the topology of weak convergence, this sequence of solutions converges to the evolution of the limiting age distribution for the monodisperse initial condition $\pi_0 = \delta_0$ in the MFFFA model. That limit obeys the transport equation up to time 1, (its gelation time), and subsequently satisfies \eqref{eq: IVP} + \eqref{eq: phi defined by theta}. 

A second obstacle is that even within the set of age-critical probability measures with mean bounded by $2$, say, the functional $\Phi(\pi)$ is not continuous with respect to the topology of weak convergence, as we can see by considering the probability measure $\pi_n$ supported on $\{0, 1, n^2\}$  such that $\pi_n(\{1\}) = 1-1/n$, $\pi_n(\{n^2\}) = 1/(n^2+n-1)$.  Each $\pi_n$ is age-critical and as $n \to \infty$ we have $\int x \,d\pi_n(x) \nearrow 2$, but $\pi_n \to \delta_1$ in the weak topology. So $\pi_n \not\to \delta_1$ in the $W_1$-topology. Writing $\theta_n$ for the normalized leading eigenfunction of $\mathcal{L}_{\pi_n}$, we have $\theta_n(1) =1$ and $\theta_n(n^2) = n + 1 - 1/n$. Thus $\Phi(\pi_n) \sim 1/n$ as $n \to \infty$, even though $\Phi(\delta_1) = 1$. 

These examples leave open the possibility that looking only at $\pi_t$ and not $\lambda_t, \theta_t$ or $\varphi(t)$, the limiting age evolution identified in Theorem~\ref{thm_convergence} may display continuous dependence on initial condition $\pi_0$ with respect to the topology of weak convergence on the space of age-critical or age-subcritical Borel probability measures on $[0,\infty)$ with positive finite mean.


\subsection{Related work on frozen percolation}
The mean-field frozen percolation model, introduced in R\'ath \cite{Rath} is very similar to the MFFF model, with the difference that when a cluster is struck by lightning its vertices are \emph{frozen}, meaning that they are no longer able to add new edges. Edges arrive only between unfrozen vertices. Using methods similar to those in \cite{RathToth}, R\'ath showed that for the asymptotic r\'egime of lightning rates $\lambda(n) \to 0$, $n \lambda(n) \to \infty$ the frozen percolation model exhibits self-organized criticality in the limit, and the vector of proportions of cluster sizes among the unfrozen vertices has a limit that satisfies the Smoluchowski/Flory coagulation equations with multiplicative kernel. 
 
 Yeo \cite{Yeo} studies frozen percolation starting from an inhomogeneous random graph with finitely many distinct types, defined by a reasonably general $k$ by $k$ kernel. Yeo \cite{Yeo} builds on R\'ath's results by studying the evolution of proportions of the $k$ types among the frozen and unfrozen vertices, showing that there is a deterministic limiting flow as the model size tends to infinity. This flow is the solution of a finite-dimensional system, of a similar nature to the system that we study in the present paper. It is driven by the leading eigenvector of a Perron-Frobenius matrix, scaled so as to keep the leading eigenvalue equal to $1$.  Yeo shows existence and uniqueness of solutions to this system. Yeo also gives a direct construction of a solution of the Smoluchowski equations from the type flow, using the total progeny distribution of the corresponding $k$-type branching processes. In the present paper we do not give a direct proof that the total progeny distribution of the multitype branching process with branching operator $\mathcal{L}_{\pi_t}$ furnishes a solution of the critical forest fire equations whenever $\pi_0$ is age-critical with finite mean and $\pi_t$ solves \eqref{eq: IVP} + \eqref{eq: phi defined by theta}. However this follows from the fact that the solution is unique and arises as the limit of MFFFA, using the stochastic limit theorems proven for MFFF in \cite{RathToth} and for MFFFA in\cite{CraneRathYeo}.  

\subsection{Open question: stability} 
 The critical forest fire equations have a unique fixed point. It is not difficult to show that the system \eqref{eq: IVP} + \eqref{eq: phi defined by theta} also has a unique constant solution $\pi_{\textup{fix}}$, which has density $\frac{1}{2}\mathrm{sech}^2(x/2)$ with respect to Lebesgue measure. The corresponding constant total burning rate and birth rate is $\varphi = \frac{1}{2}$ and the normalized leading eigenfunction of $\mathcal{L}_\pi$ is $\theta(x) = 2\tanh(x/2)$. A curious fact about the constant solution is that the age-specific mortality modulus for an individual is equal to its quantile in the age distribution. 

 It is noted in \cite{CraneFreemanToth} that for $t > t_{gel}$, a consequence of the critical forest fire equations is that 
$$ \varphi(t) - \frac{1}{2} = \frac{d}{dt} \sum_{k=1}^\infty \frac{1}{k} v_k(t)\,,$$ so that the long-time average value of $\varphi$ is $1/2$, for any solution. However, it is not known whether the unique fixed point of the critical forest fire equations is either locally or globally attractive.  Restricting attention to cluster size distributions that arise as local limits of age-driven inhomogeneous random graphs, we may instead ask about the local or global attractiveness of $\pi_{\textup{fix}}$ for the system~\eqref{eq: IVP} + \eqref{eq: phi defined by theta}. Our hope is that it may be easier to find a Liapounov function for the age evolution equations than for the critical forest fire equations.  We tentatively conjecture on the basis of numerical experiments that $W_1(\cdot, \pi_{\textup{fix}})$ is a global Liapounov function for the system~\eqref{eq: IVP} + \eqref{eq: phi defined by theta}.
 
\section{Proof of Theorem~\ref{T: well-posedness theorem}}\label{S: main proof}

\subsection{Alternative formulation of the differential equation}
Since equation~\eqref{eq: IVP} is a modification of the transport equation, it will be convenient to use test functions that are travelling waves. 
\begin{lemma}\label{L: travelling wave 1}
Equation~\eqref{eq: IVP meaning} implies that for all $f \in C_0^1(\mathbb{R})$ we have
\begin{equation}\label{eq: wave test function}
\frac{d}{dt} \int f(s-t)\,d\pi_t(s) =
 -\int f(s-t)\varphi(t)\theta_t(s)\,d\pi_t(s) + \varphi(t) f(-t)\,.
\end{equation}
\end{lemma}
\begin{proof}
Equation~\eqref{eq: wave test function} is equivalent to the integral formulation
\begin{multline}\label{eq: travelling wave integral form} \int f(s-t)\,d\pi_t(s) - \int f(s)\,d\pi_0(s) \\= \int_0^t \left( -\int f(s-u)\varphi(u)\theta_u(s)\,d\pi_u(s) + \varphi(u)f(-u)\right)\,du\,. \end{multline}
This holds trivially for $t=0$. Since $f \in C_0^1(\mathbb{R})$ we may differentiate both sides of~\eqref{eq: travelling wave integral form} with respect to $t$ and~\eqref{eq: IVP meaning} says that the derivatives agree.
\end{proof}

\subsection{Uniform integrability of solutions}

\begin{lemma}\label{L: mean age stays finite} For any solution of~\eqref{eq: IVP} over $t \in [0,T]$, such that $\pi_0 \in \mathcal{P}_1([0,\infty))$, we have
$$\int_0^\infty \sup_{t \in [0,T]} \pi_t([x,\infty)) \,dx < \infty\,.$$
and the family $\{\pi_t: t \in [0,T]\}$ is uniformly integrable. Furthermore, for every $0 \le t \le T$, $$\int x\,d\pi_t(x) \le t + \int x\,d\pi_0(x)\,.$$
\end{lemma}
\begin{proof} 
Let $x \ge 0$. For every $f \in C_0^1(\mathbb{R})$ with values in $[0,1]$ and supported on $[x,\infty)$ we have from~\eqref{eq: wave test function} that for all $t \ge 0$ $$ \int f(s-t)\,d\pi_t(s) \le \int f(s) \,d\pi_0(s) \le \pi_0([x,\infty))\,.$$ $f(\cdot - t)$ may approximate the indicator function of $[x,\infty)$ from below, so $$ \pi_t([x+t,\infty)) \le \pi_0([x,\infty))\,.$$
Hence $$\sup_{t \in [0, T]} \pi_t([x,\infty)) \le \begin{cases} 1 & \text{if $x \le T$},\\ \pi_0([x-T,\infty)) & \text{if $x \ge T$.}\end{cases}$$
Uniform integrability follows, as does
\begin{eqnarray*} \int_0^\infty \sup_{t \in [0, T)} \pi_t([x,\infty))\,dx  &\le& T + \int_0^\infty \pi_0([y,\infty))\,dy \\ &=& T + \int x \,d\pi_0(x)\,<\, \infty\,.\end{eqnarray*}
For any $0 \le t \le T$,
\begin{eqnarray*}\int x\, d\pi_t(x) &\le & t + \int (x-t)^+  d\pi_t(x) =  t + \int_{0}^\infty \pi_t([x+t,\infty))\,dx \\ & \le & t + \int_0^\infty \pi_0([x,\infty)) = t + \int x \,d\pi_0(x)\,.
\end{eqnarray*}
\end{proof}

\subsection{The metric $W_1$}\label{SS: W_1 metric}
 
In this section we give a quick summary of the $W_1$ metric, followed by a first application to bound the speed of $\pi_t$ in the $W_1$ metric. For a fuller account of the properties of $W_1$, see Rachev~\cite{Rachev} or Dudley~\cite[\S11.8]{Dudley}.
 
Let $(X,d)$ be a complete separable metric space. Denote by $\mathcal{P}_1(X,d)$ the space of Borel probability measures $\mu$ on $(X,d)$ s.~t.~$\int d(x,x_0)\,d\mu(x) < \infty$ for some (hence any) basepoint $x_0 \in X$. For $\mu, \nu \in \mathcal{P}_1(X,d)$, consider the set $\mathcal{C}(\mu,\nu)$ of all Borel probability measures on $X \times X$ with marginals $\mu$ and $\nu$, i.e. the set of couplings of $\mu$ and $\nu$. Define
$$W_1(\mu, \nu) = \inf_{m \in \mathcal{C}(\mu,\nu)} \int d(x,y) \,dm(x,y)\,.$$
Then $W_1$ is a complete separable metric on $\mathcal{P}_1(X,d)$.

  For a measure $\mu \in \mathcal{P}_1(X,d)$ and a sequence $\left(\mu_n\right)$ of measures in $\mathcal{P}_1(X,d)$, $W_1(\mu_n,\mu) \to 0$ if and only if both $\mu_n \to \mu$ weakly (against bounded continuous functions) and $\int d(x,x_0) d\mu_n \to \int d(x,x_0) d\mu$.  Note the latter condition does not depend on the choice of basepoint, for the function $x \mapsto d(x,x_0) - d(x,x_1)$ is bounded and continuous.

The $W_1$ metric has a useful dual formulation, originally due to Kantorovich and Rubinstein for compact metric spaces, and generalized to separable metric spaces by Dudley and de Acosta. Denote by $\mathrm{Lip}^1(X,d)$ the space of real-valued functions of $X$ that are $1$-Lipschitz with respect to the metric $d$. Then for any $\mu,\nu \in \mathcal{P}_1(X,d)$, we have 
\begin{equation} \label{eq: W1 dual formulation} W_1(\mu,\nu)  = \sup\left\{ \left|\int f d\mu - \int f d\nu\right|: \;f \in \mathrm{Lip}^1(X,d)\right\}\,.\end{equation}   
For example, if $(X,d)$ is $\mathbb{R}$ with the Euclidean metric then taking $f(x) = x$ shows that $\int x\,d\pi(x)$ is a $1$-Lipschitz functional of $\pi$ with respect to $W_1$. We write $C_0^1(\mathbb{R})$ for the space of continuously differentiable and compactly supported real-valued functions on $\mathbb{R}$.

\begin{lemma}\label{L: restricted duality}
For any $\mu, \nu \in \mathcal{P}_1(\mathbb{R})$ we have
\begin{equation} \label{eq: W1 dual formulation, restricted} W_1(\mu,\nu)  = \sup\left\{ \left|\int f d\mu - \int f d\nu\right|: \;f \in \mathrm{Lip}^1(\mathbb{R}) \cap C_0^1(\mathbb{R})\right\}\,.\end{equation}   
\end{lemma}
\begin{proof}
It suffices to show that for any $f \in \mathrm{Lip}^1(\mathbb{R})$ we may find a sequence $f_n$ of functions in $\mathrm{Lip}^1(\mathbb{R}) \cap C_0^1(\mathbb{R})$ such that
$$\left| \int f_n \,d\mu - \int f_n\,d\nu\right| \to \left| \int f\,d\mu - \int f\,d\nu\right|\,.$$
We do this in two steps. First, approximate $f$ by a sequence of functions $g_n$ with support in $[-2n-|f(0)|,2n+|f(0)|]$. To do this, let $g_n$ be the $1$-Lipschitz function that agrees with $f$ on the interval $[-n,n]$, which interpolates linearly between $f(n)$ at $n$ and $0$ at $n+|f(n)|$ and between $f(-n)$ at $-n$ and $0$ at $-n-|f(-n)|$, and which vanishes outside $[-n-|f(-n)|, n+|f(n)|]$. Then $f-g_n$ is $2$-Lipschitz and vanishes on $[-n,n]$, so
$$\left|\int f - g_n \,d\mu\right| \le 2\int_n^\infty   \mu([x,\infty))\,dx \to 0 \quad\text{as $n \to \infty$}\,,$$ and likewise for $\nu$. Second, let $f_n$ be the convolution of $g_n$ with a smooth bump function that is supported in $[-1/n,1/n]$ and has integral equal to $1$. Then $f_n$ is $1$-Lipschitz, and $\|f_n - g_n\|_\infty \le   1/n$ because $g_n$ is $1$-Lipschitz. Hence 
$$\left|\int f_n - g_n \,d\mu\right| \le 1/n \,,$$ and likewise for $\nu$, since $\mu$ and $\nu$ are probability measures.
\end{proof}

For probability measures on $\mathbb{R}$ we may express the $W_1$ metric in terms of the cumulative distribution functions of the measures. If $F_\mu$ and $F_\nu$ denote the cumulative distribution functions of $\mu$ and $\nu$, then $W_1(\mu,\nu)$ is the $L^1$ norm of $F_\mu - F_\nu$.  For example, letting $\delta_0$ denote the atom of unit mass at $0$, we have
$$\int x \,d\pi(x) = W_1(\pi,\delta_0)\,.$$
The monotonic quantile coupling of $\mu$ and $\nu$ achieves the infimum in the definition of $W_1(\mu,\nu)$, but is often useful to bound $W_1$ above using other couplings. 

We let $\tau_r$ denote translation by $r$, operating on measures: for any Borel set $A$,
$$ \tau_r(\mu)(A)  = \mu(\{x: x+r \in A\})\,.$$ For any $\mu \in \mathcal{P}_1(\mathbb{R})$ we have $W_1(\mu, \tau_r(\mu)) = r$. The coupling of $\mu$ and $\tau_t(\mu)$ that translates all of the mass of $\mu$ by $r$ is optimal. Thus any solution of the pure transport equation on $\mathcal{P}_1([0,\infty))$ moves at speed $1$ with respect to $W_1$.

\begin{lemma}\label{L: speed bound}
 Let $\pi_t: t \in [0,T]$ be any solution of \eqref{eq: IVP} taking values in $\mathcal{P}_1([0,\infty))$. Then for $0 \le u \le v\ \le T$, 
 \begin{equation} \label{eq: W1 speed} W_1(\pi_u,\pi_v) \le |v-u| + \int_u^v \varphi(t)\int s \theta_t(s) \,d\pi_t(s)\,dt \end{equation}
\end{lemma}
\begin{proof}
By Lemma~\ref{L: restricted duality},  it suffices to show that for any test function $f \in \mathrm{Lip}^1([0,\infty)) \cap C_0^1(\mathbb{R})$,  $|\int f \,d\pi_u - \int f\,d\pi_v|$ is bounded by the RHS of \eqref{eq: W1 speed}. This follows when we observe
\begin{eqnarray*} \left| \frac{d}{dt} \int f(s) \,d\pi_t(s)\right| & \le &  \int |f'(s)|\,d\pi_t(s) \; +\; \left|\int (f(s) - f(0)) \varphi(t) \theta_t(s)\,d\pi_t(s)\right| \\ & \le & 1 + \int s \varphi(t) \theta_t(s)\,d\pi_t(s)\,.\end{eqnarray*}
\end{proof}
It is not obvious a priori that the inner integral on the RHS of \eqref{eq: W1 speed} is finite. After some work we will show that $\theta_t$ may be uniformly bounded in terms of the tail of $\pi_t$, and from this we will deduce that the speed of $\pi_t$ with respect to $W_1$ may be bounded for $t$ in some time interval $[0,\epsilon)$ where $\epsilon$ and the bound on the speed both depend on $\pi_0$. 

\subsection{A useful system of neighborhoods}
Define the following subsets of $\mathcal{P}_1([0,\infty))$:
$$\mathcal{P}_{r,c} := \left\{\pi \in \mathcal{P}_1([0,\infty))\,:\, \int \mathbf{1}(x \ge r)\,x \,d\pi(x) < c\right\}\,.$$
For each fixed $c > 0$, the sets $\mathcal{P}_{r,c}\,:\, r \in (0,\infty)$ are nested open sets in the $W_1$ topology, which cover $\mathcal{P}_1([0,\infty))$. If $\pi \in \mathcal{P}_{r,c}$ then $\int x\,d\pi(x) \le r + c$.

\begin{lemma}\label{L: metric nestedness}
Let $r' > r \ge 0$. If $\pi \in \mathcal{P}_{r,c}$ then for any $r' > r$ and any $\tilde{\pi} \in \mathcal{P}_1([0,\infty))$ we have
$$ \tilde{\pi} \in \mathcal{P}_{r',\, c + (r'/(r'-r)) W_1(\tilde{\pi},\pi)}\,.$$
\end{lemma}
\begin{proof}
$$ \int \mathbf{1}(x > r') x \,d\tilde{\pi}(x)  \; \le \; \int \frac{r'}{r'-r}(x-r) \mathbf{1}(r < x \le r') + x \mathbf{1}(x > r')\,d\tilde{\pi}(x)\,.$$
The integrand on the right is $\frac{r'}{r'-r}$-Lipschitz so $\int \mathbf{1}(x > r') x \,d\tilde{\pi}(x)$ is bounded above by
\begin{eqnarray*} &  & \frac{r'}{r'-r}W_1(\tilde{\pi},\pi) + \int \frac{r'}{r'-r}(x-r) \mathbf{1}(r < x \le r') + x \mathbf{1}(x > r')\,d\pi(x)\\ & \le & \frac{r'}{r'-r}W_1(\tilde{\pi},\pi) + \int \mathbf{1}(x > r) x \,d\pi(x)\;\le\; \frac{r'}{r'-r} W_1(\tilde{\pi},\pi)\,+ c\,. \end{eqnarray*}
\end{proof}

\subsection{Properties of the operator $\mathcal{L}_\pi$}

Let $\pi \in \mathcal{P}_1([0,\infty))$. Assume that $\pi \neq \delta_0$ so that $\mathcal{L}_\pi$ is not the zero operator. Denote the inner product and norm on the Hilbert space $L^2(\pi)$ by $\langle \cdot,\cdot\rangle_\pi$ and $\|\cdot\|_\pi$. For any $f \in L^2(\pi)$ we have
\begin{eqnarray*} \|\mathcal{L}_\pi f\|^2_\pi & = & \int \left(\int (x \wedge y) f(y)\,d\pi(y)\right)^2 \,d\pi(x) \\ & \le & \int  \left(\int (x \wedge y)^2\,d\pi(y)\right)\left(\int f(y)^2 \,d\pi(y)\right) \,d\pi(x)\\
& \le & \|f\|_\pi^2 \int\!\int xy\,d\pi(y)\,d\pi(x) = \|f\|_\pi^2 \left(\int x\,d\pi(x)\right)^2\,.
\end{eqnarray*}
Thus $\mathcal{L}_\pi$ is a bounded operator with operator norm $\|\mathcal{L}_\pi\| \le \int x \,d\pi(x)$. In fact $\mathcal{L}_\pi$ is a Hilbert-Schmidt integral operator, with Hilbert-Schmidt norm $$\|\mathcal{L}_\pi\|_{\mathrm{HS}} = \left(\int\int(x \wedge y)^2 \,d\pi(x)\,d\pi(y)\right)^{1/2} \le \int x \,d\pi(x)\, < \infty\,.$$ It follows (see e.g. \cite[Theorem VI.22]{reed_simon})
that $\mathcal{L}_\pi$ is a compact operator and $\|\mathcal{L}_\pi\| \le \|\mathcal{L}_\pi\|_{\mathrm{HS}}$. The following facts about $\mathcal{L}_\pi$ are proved in \cite[Lemma 4.3]{CraneRathYeo}.
\begin{lemma}\label{L: operator properties}{\quad}
\begin{enumerate}[(i)]
\item \label{L:self_adjoint_pos} $\mathcal{L}_\pi$ is a positive semidefinite and compact self-adjoint  operator. 
\item \label{L:Lipsch} Each element of the image of $\mathcal{L}_\pi$ is represented by a Lipschitz function, and $\mathcal{L}_\pi$ maps non-negative functions to increasing non-negative functions.
\item \label{L:principal}
$\mathcal{L}_\pi$ has a simple principal eigenvalue $\lambda$ satisfying $0 < \lambda = \|\mathcal{L}_\pi\|$ (i.e. the eigenspace associated to $\lambda$ is one-dimensional).
\item \label{L:eigen} There exists a unique eigenfunction $\theta\in L^2(\pi)$ for which $\mathcal{L}_\pi \theta=\lambda \theta$ and $\int \theta(x)\mathrm{d}\pi(x)=1$. We may identify $\theta$ with its increasing Lipschitz-continuous representative defined for each $y \ge 0$ by
\begin{equation}\label{eq: theta rep defn} \lambda \theta(y) = \int (x \wedge y) \theta(x) \,d\pi(x)\,.\end{equation}
\end{enumerate}
\end{lemma}

\begin{remark}In fact \cite[Thm 4.6]{Aleksandrov} shows (after a change of variables) that $\mathcal{L}_\pi$ is a bounded operator on $L^2(\pi)$ if and only if $\pi([x,\infty)) = O(1/x)$ as $x \to \infty$, it is compact if and only if $\pi([x,\infty)) = o(1/x)$ as $x \to \infty$, and it belongs to the trace class if and only if $\int x\,d\pi(x) < \infty$. (See also \cite[Example 17.6]{BollobasJansonRiordan}.)
 Thus for $\pi \in \mathcal{P}_1([0,\infty))$, $\mathcal{L}_\pi$ has discrete spectrum and the sum of the eigenvalues is the trace, which can be shown to be $\int x\,d\pi(x)$. The sum of the squared eigenvalues is $\|\mathcal{L}_\pi\|_{HS}^2$. One can show that the spectral gap of $\mathcal{L}_\pi$ is at least $1/\lim_{x \to \infty} \theta(x)$. We will not use these properties in the proof of Theorem~\ref{T: well-posedness theorem}. 
\end{remark}

\subsection{Correspondence between $\theta$ and $\pi$}
\label{SS: theta-pi correspondence}

The following results explain that for $\pi \in \mathcal{P}_1((0,\infty))$, the eigenvalue problem for $\mathcal{L}_\pi$ is a generalized Sturm-Liouville problem, and that the principal solution $(\lambda, \theta)$ determines $\pi$. Moreover, $\theta$ is bounded, and the explicit bound on $\theta$ in terms of the tail of $\pi$ will be crucial in later sections.

\begin{lemma}[$\theta$ determines $\pi$]
For $\pi \in \mathcal{P}_1((0,\infty))$, let $\theta$ be the representative of the normalized principal eigenfunction of $\mathcal{L}_\pi$ with eigenvalue $\lambda$ defined by equation~\eqref{eq: theta rep defn}. Then $\theta$ is Lipschitz, concave and increasing and $\theta(x) = o(x)$ as $x \to \infty$. Moreover, $\theta$ determines $\pi$ since for each $0 < a < b < \infty$ we have \begin{equation}\label{eq: theta-pi correspondence integral form} \int \theta(x) \mathbf{1}(a < x < b) \,d\pi(x) = \lambda( \theta_r'(a) - \theta_l'(b))\,,\end{equation} 
where $\theta_l'$ and $\theta_r'$ denote the left and right derivatives of $\theta$ respectively. 
\end{lemma}
We can abbreviate~\eqref{eq: theta-pi correspondence integral form} by saying that the pair $(\lambda, \theta)$ determines $\pi$ via 
\begin{equation} d\pi(x) = \frac{-\lambda \theta''(x)}{\theta(x)}\,dx\,, \label{eq: theta-pi correspondence}\end{equation}
where $\theta''$ is interpreted in its distributional sense.
\begin{proof}
We may differentiate equation~\eqref{eq: theta rep defn} with respect to $y$ to obtain left and right derivatives
\begin{equation}\label{eq: left derivative} \lambda \theta_l'(y) = \int \mathbf{1}(x \ge y) \theta(x) \,d\pi(x)\,, \text{ for $y > 0$,} \end{equation} 
\begin{equation}\label{eq: right derivative} \lambda \theta_r'(y) = \int \mathbf{1}(x > y) \theta(x) \,d\pi(x)\,,\text{ for $y \ge 0$.}\end{equation}
Since these are both non-negative decreasing functions of $y$, we see that $\theta$ is increasing and concave. We have $\theta(0) = 0$ and $\theta'_r(0)  = \lambda^{-1} < \infty$ so $\theta$ is Lipschitz with constant $\lambda^{-1}$.
Letting $y \to \infty$  we see that $\lim_{y \to \infty} \theta_l'(y) = \lim_{y \to \infty} \theta'_r(y) = 0$, which implies that $\theta(x) = o(x)$ as $x \to \infty$.  Taking the difference of~\eqref{eq: left derivative} and~\eqref{eq: right derivative} we obtain
$$ \int \theta(x) \mathbf{1}(a < x < b) \,d\pi(x) = \lambda( \theta_r'(a) - \theta_l'(x))\,.$$
\end{proof}

\begin{proposition} \label{P: theta-pi correspondence}
 Let $\theta: [0,\infty) \to [0,\infty)$ be a continuous concave increasing function such that $\theta(0) = 0$, $\theta(x) > 0$ for $x > 0$, and $\theta(x) = o(x)$ as $x \to \infty$, and let $\lambda > 0$. Then there is a non-trivial locally finite positive Borel measure $\pi$ on $(0,\infty)$ defined by \eqref{eq: theta-pi correspondence}. Let $\theta'_r$ denote the right-derivative of $\theta$. Then
\begin{enumerate}
 \item $\int \mathbf{1}_{x > 1} x \,d\pi(x) < \infty$ if and only if $\theta$ is bounded,
 \item $\int \mathbf{1}_{ x < 1} x \, d\pi(x) < \infty$ if and only if $\theta'_r$ is bounded, i.e.~$\theta_r'(0) < \infty$.
\end{enumerate}
When both of these conditions hold then $\mathcal{L}_\pi$ is a positive semidefinite self-adjoint Hilbert-Schmidt operator on $L^2(\pi)$ with simple principal eigenvalue $\lambda$ and $\theta$ represents a principal eigenfunction of $\mathcal{L}_\pi$.
\end{proposition}

Note that the measure $\pi$ determined by $\theta$ in Proposition~\ref{P: theta-pi correspondence} need not be a probability measure; in fact it need not be finite. Also, not every $\pi \in \mathcal{P}([0,\infty))$ arises from a $(\lambda, \theta)$ pair as in Proposition~\ref{P: theta-pi correspondence} solving~\eqref{eq: theta e'vec}. For example, if $\pi = \sum_{i=1}^\infty p_i \delta_{a_i}$ where $\sum_{i=1}^\infty p_i = 1$ but $\sup_i (p_i a_i) = \infty$ then from~\eqref{eq: theta rep defn} we find
$$ \lambda \theta(a_i)  \ge p_i a_i \theta(a_i)\,, $$
which cannot hold for all $i$. Nevertheless, Proposition~\ref{P: theta-pi correspondence} establishes a one-to-one correspondence between non-zero elements of $\mathcal{M}_1^{+}((0,\infty))$, up to multiplication by positive scalars, and bounded Lipschitz concave increasing functions $\theta: [0,\infty) \to [0,\infty)$ with $\theta(0) =0$ and $\theta_r'(0) < \infty$, also up to multiplication by positive scalars.

\begin{proof}
By concavity the derivative of $\theta$ exists Lebesgue-almost everywhere. The right-derivative $\theta_r'$ exists everywhere on $[0,\infty)$, possibly taking the value $\infty$ at $0$ but finite everywhere else. $\theta'_r$ is a decreasing right-continuous function on $[0, \infty)$.  Likewise, the left-derivative $\theta_l'$ exists and is finite everywhere on $(0,\infty)$ and defines a decreasing left-continuous function there. We have $\lim_{y \nearrow x} \theta_r'(y) = \theta_l'(x) \ge \theta_r'(x) = \lim_{y \searrow x} \theta_l'(y)$ for all $x > 0$, and $\theta_l'(x) = \theta_r'(x)$ exactly where $\theta'(x)$ exists, which occurs Lebesgue-a.e.


Since $\theta$ is continuous and $\theta(x) > 0$ for all $x > 0$, the equation~\eqref{eq: theta-pi correspondence integral form} determines a positive Borel measure $\pi$, which has an atom of mass $(\theta_l'(x) - \theta_r'(x))/\theta(x)$ at each point $x$ where $\theta_l'(x) \neq \theta'_r(x)$, and no other atoms.

$\theta$ is concave increasing and $\theta(0) = 0$, so $0 \le \theta_r'(x) \le \theta_l'(x) \le \theta(x)/x$ for all $x > 0$, and since $\theta(x) = o(x)$ as $x \to \infty$, we have $\theta_l'(x) \to 0$ and $\theta_r'(x) \to 0$ as $x \to \infty$.  Letting $b \to\infty$ in~\eqref{eq: theta-pi correspondence integral form} we obtain
\begin{equation}\label{eq: right derivative of theta} \int \theta(x)\mathbf{1}(a < x)\,d\pi(x) = \lambda \theta_r'(a), \end{equation}
and adding on the contribution from any atom of $\pi$ at $a$ we also have
\begin{equation}\label{eq: right derivative of theta 2} \int \theta(x)\mathbf{1}(a \le x)\,d\pi(x) = \lambda \theta_l'(a), \end{equation}
Integrating equation~\eqref{eq: right derivative of theta} with respect to $a$ from $0$ to $s$ and using Fubini's theorem we obtain for each $s \ge 0$ that
\begin{eqnarray} \notag \int (x \wedge s) \theta(x)\,d\pi(x) & = & \int_0^s \int \theta(x) \mathbf{1}(a < x)\,d\pi(x) \,da\\ & = & \lambda \int_0^s \theta_r'(x)\,da = \lambda(\theta(s) - \theta(0)) = \lambda \theta(s)\,.\label{eq: eigenvalue equation with s} \end{eqnarray} 
Define $\theta(\infty) = \lim_{x \to \infty} \theta(x)$, possibly taking the value $\infty$. Then taking the limit as $s \to \infty$ in~\eqref{eq: eigenvalue equation with s}, we have by monotone convergence that
$$ \int x \theta(x)\,d\pi(x) = \lambda \theta(\infty)$$
so
$$ \int \mathbf{1}_{x > 1} x\,d\pi(x) \le \int \mathbf{1}_{x > 1} \frac{\theta(x)}{\theta(1)} x\,d\pi(x) \le \lambda \frac{\theta(\infty)}{\theta(1)}\,.$$ 
and
$$ \int \mathbf{1}_{x \le 1} x\,d\pi(x) \le \int \mathbf{1}_{x \le 1} \frac{\theta(x)}{\theta(1)}\,d\pi(x) = \lambda \frac{\theta'_r(0) - \theta'_r(1)}{\theta(1)}\,.$$
Therefore $\int \mathbf{1}_{x > 1} x\,d\pi(x) < \infty$ if $\theta$ is bounded and $\int \mathbf{1}_{x < 1} x\,d\pi(x) < \infty$ if $\theta_r'(0) < \infty$.

The converse parts of claims 1 and 2 are closely related. To see this consider the function $\eta: (0,\infty) \to (0,\infty)$ defined by $\eta(x):=x \theta(1/x)$. It is a simple exercise to check that the transformation $ \theta \mapsto \eta$ is an involution on the set of continuous increasing concave functions on $(0,\infty)$ such that $\theta(x) \to 0$ as $x \to 0$ and $\theta(x) = o(x)$ as $x \to \infty$.


Let $\nu$ be the measure defined distributionally by $d\nu = - \lambda \frac{\eta''}{\eta}\,dx$, as in~\eqref{eq: theta-pi correspondence}.
As an equality of distributions, we have $$\eta''(x) dx = \frac{\theta''(1/x)}{x^3}\, dx\,.$$ Making the substitution $x = 1/y$ we have
\begin{eqnarray*} \int \mathbf{1}_{y < 1}\, y \,d\pi(y) & = & -\lambda\int \mathbf{1}_{y < 1}\, y\, \frac{\theta''(y)}{\theta(y)}\,dy  = -\lambda\int \mathbf{1}_{x > 1}\, \frac{1}{x} \frac{\theta''(1/x)}{\theta(1/x)} \frac{dx}{x^2} \\ & = & -\lambda\int \mathbf{1}_{x > 1}\, x \frac{\eta''(x)}{\eta(x)} \,dx  =  \int \mathbf{1}_{x > 1}\, x\,d\nu(x)\,. \end{eqnarray*}
We have $\lim_{x \to \infty} \theta(x)  = \lim_{x \to 0} \eta_r'(x)$ and $\lim_{x \to \infty} \eta(\infty) = \lim_{x \to 0} \theta_r'(0)$. 
So to prove claim 2 it suffices to prove claim 1 for $\theta$ and $\pi$ and apply it to $\eta$ and $\nu$.
A similar calculation shows that the finite size-biased measures that we may informally write as $y \,d\pi(y)$ and $x\,d\nu(x)$ are each other's pushforward under the transformation $y = 1/x$.

Suppose that $\int \mathbf{1}_{x > 1} x \,d\pi(x) < \infty$. We must prove that $\theta(\infty) < \infty$. Observe that $\theta$ is eventually constant if and only if $\pi$ has compact support.  So we may suppose that $\theta$ is not eventually constant, and hence $\theta_r'(x) > 0$ everywhere.
Because $\theta$ is concave, $\theta(0) = 0$, and $\theta(x) > 0$ for every $x > 0$, we have $$ 0 < \frac{x\theta_r'(x)}{\theta(x)} \le \frac{x\theta_l'(x)}{\theta(x)} \le 1 \quad \text{ for every  $x > 0$.}$$ 
The example $\theta(x) \equiv x^{1/2}$ shows that for a general concave increasing function satisfying $\theta(x) = 0$ and $\theta(x) = o(x)$ as $x \to \infty$ we need not have $x \theta_l'(x)/\theta(x) \to 0$ as $x \to \infty$. However, the condition $\int \mathbf{1}(x > 1) x \,d\pi(x) < \infty$ implies $x \theta_l'(x)/\theta(x) \to 0$ as $x \to \infty$. Indeed, since  $\theta$ is concave and $\theta(0) =  0$, $\theta(x)/x$ is decreasing in $x$ and therefore if $0 < x < y$ then $x\theta(y)/\theta(x) \le y$. Using equation~\eqref{eq: right derivative of theta 2} we have
\begin{eqnarray} \notag \lambda \frac{x \theta_l'(x)}{\theta(x)} & = &  \int\! \mathbf{1}(y \ge x) \frac{x}{\theta(x)} \theta(y) \,d\pi(y) \\ & \le & \int \!\mathbf{1}(y \ge x)\, y \,d\pi(y)\,.\end{eqnarray}
It follows that $x\theta_l'(x)/\theta(x) \to 0$ as $x \to \infty$. Likewise we have
\begin{equation}\label{eq: quotient bound} \lambda \frac{x\theta_r'(x)}{\theta(x)} \le \int \mathbf{1}(y > x) y \,d\pi(y)\,.\end{equation}

The explicit inequality that will show that $\theta$ is in fact bounded is proved by a simple integration by parts, in the case where $\theta$ is twice continuously differentiable. We handle the general case by approximating $\theta$ by smooth concave increasing functions, but it requires a little care. 

 Fix any $0 < x < y$. We claim we may approximate $\theta$ uniformly from below on $[x,y]$ by a sequence of twice continuously differentiable concave increasing functions $\theta_n: [x,y] \to [\theta(x),\theta(y)]$ such that $\theta_n(x) = \theta(x)$, $\theta_n(y) = \theta(y)$, and $\sup_{[x,y)} \theta_n'/\theta_r'  \to 1$ as $n \to \infty$. Then \begin{equation}\label{eq: approximation}\sup_{s \in [x,y)} \frac{\theta_n'(s) \theta(s)}{\theta_r'(s) \theta_n(s)} \to 1 \quad\text{as $n \to \infty$}\,.\end{equation} To achieve such an approximation, we first approximate $\theta$ from below by a sequence of piecewise linear convex functions $\tilde{\theta}_n$, where $\tilde{\theta}_n$ interpolates linearly between the values of $\theta$ at any sequence of points $s_i \in [x,y]$ where $s_0 = x$, $s_n = y$ and $$\theta_r'(s_i) \le \theta_r'(x)^{(n-i)/n}\theta_l'(y)^{i/n} \le \theta_l'(s_i) \quad \text{for $i \in \{1, \dots, n-1\}$.}$$
  Then $\sup_{[x,y]} \tilde{\theta}_n'/\theta_r' \le \left(\theta_l'(y)/\theta_r'(x)\right)^{1/n} \to 1$ as $n \to \infty$. It follows that $\tilde{\theta}_n \to \theta$ pointwise on $[x,y]$. It is simple to approximate $\tilde{\theta}_n$ from below by a smooth concave function $\theta_n$ such that $\theta_n(x) = \tilde{\theta}_n(x) = \theta(x)$, $\theta_n(y) = \tilde{\theta}_n(y) = \theta(y)$, and $\theta_n \le (1+1/n) \tilde{\theta}_n$, $\theta_n' \le (1+1/n)\tilde{\theta}_n'$ on $[x,y]$. 

  Let $\pi_n$ be the measure on $[x,y]$ defined by $$d\pi_n(s) = -\lambda \frac{\theta_n''(s)}{\theta_n(s)}\,ds\, = -\frac{\lambda}{\theta_n(s)} d\theta_n'(s)\,.$$
It follows from the fact that each $\theta_n$ is concave and $\theta_n \nearrow \theta$ pointwise that at every point $s \in (x,y)$ where $\theta'(s)$ exists, we have $\theta_n'(s) \to \theta'(s)$ as $n \to \infty$. Thus $\theta_n'(s) \to \theta'(s)$ for a.e.~$s \in (x,y)$. We also have $\theta_n'(x) \to \theta_r'(x)$, $\theta_n'(y) \to \theta_l'(y)$ as $n \to \infty$. It follows that $\pi_n$ converges weakly to the restriction of $\pi$ to $(x,y)$. Since $$\theta_n'(x) \le\theta_r'(x) \le \theta(x)/x = \theta_n(x)/x\,,$$ and $\theta_n$ is concave, $s/\theta_n(s)$ is increasing in $s$ for each $n$. 


Now choose $0 < c < 1$ and find $x_0 > 0$ depending on $c$ such that $$\int \mathbf{1}(x > x_0) x \,d\pi(x) < \lambda c\,.$$ 
Then we have $\sup_{s \ge x_0} s\theta'_r(s)/\theta(s) < c$, by~\eqref{eq: quotient bound}.  Now assume $x_0 \le x < y$. It follows from~\eqref{eq: approximation} that for $n$ sufficiently large,
\begin{equation}\label{eq: useful sup}\sup_{s \in [x,y)} \frac{s\theta_n'(s)}{\theta_n(s)} < c\,.\end{equation} 
Assuming~\eqref{eq: useful sup} holds,  integrating by parts we have 
\begin{eqnarray*} \lambda^{-1}\int\! s\,d\pi_n(s) &=&  \int_x^y \frac{-s \theta_n''(s)}{\theta_n(s)}\,ds \\ & =  & \frac{x \theta_n'(x)}{\theta_n(x)}  - \frac{y \theta_n'(y)}{\theta_n(y)} + \int_x^y \frac{\theta_n'(s)}{\theta_n(s)}\left(1 - \frac{s\theta_n'(s)}{\theta_n(s)}\right)\,ds 
\\ & \ge & -\frac{y \theta_n'(y)}{\theta_n(y)} + (1-c) \int_x^y \frac{\theta_n'(s)}{\theta_n(s)}\,ds 
\\ &= & -\frac{y \theta_n'(y)}{\theta_n(y)} + (1-c)\log \theta_n(y) - (1-c)\log \theta_n(x)
\\ & = & -\frac{y \theta_n'(y)}{\theta_n(y)} + (1-c)\log \theta(y) - (1-c)\log \theta(x)\,.\end{eqnarray*}
Taking the limit as $n \to \infty$, we have
$$ c > \lambda^{-1} \int \mathbf{1}(x < s < y)\,s \,d\pi(s) \ge -\frac{y\theta_l'(y)}{\theta(y)} + (1-c)\log \frac{\theta(y)}{\theta(x)}\,.$$
Taking the limit as  $y \to \infty$, we find 
$$ c \ge (1-c)\log\frac{\theta(\infty)}{\theta(x_0)}\,,$$
or equivalently
\begin{equation}\label{eq: explicit theta bound in proof} \theta(\infty) \le \theta(x_0) e^{c/(1-c)}\,.\end{equation}
This completes the proof of claim 1, and consequently also claim 2. 

Now suppose that $\theta$ is bounded and $\theta_r'(0) < \infty$, so that $\int x\,d\pi(x) < \infty$. The definition of $\mathcal{L}_\pi$ and most of the proof of Lemma~\ref{L: operator properties} do not rely on the total mass of $\pi$ being equal to $1$. In fact, when $\pi$ is any positive Borel measure on $(0,\infty)$ such that $\int x \,d\pi(x) < \infty$, we may still define $\mathcal{L}_\pi$ as a self-adjoint Hilbert-Schmidt integral operator on $L^2(\pi)$, and it still has a simple principal eigenvalue $\tilde{\lambda}$ with an associated non-negative  eigenfunction $\tilde{\theta}$ on $(0,\infty)$.  The only place where the finiteness of $\pi$ is used in the proof of Lemma~\ref{L: operator properties} is to ensure that the chosen normalization is possible. Equation~\eqref{eq: eigenvalue equation with s} says that $\theta$ is also an eigenfunction of $\mathcal{L}_\pi$, with eigenvalue $\lambda$. Since $\theta$ is non-negative and the eigenfunctions of $\mathcal{L}_\pi$ belonging to distinct eigenvalues are orthogonal in $L^2(\pi)$, we must have $\lambda = \tilde{\lambda}$. Since the principal eigenvalue of $\lambda$ is simple, $\theta$ agrees with $\tilde{\theta}$ up to multiplication by a positive scalar.

\end{proof}

\subsection{Explicit bounds on $\lambda$ and $\theta$}
In this section we suppose that $\pi \in \mathcal{P}_1([0,\infty))$ is a positive Borel probability measure on $[0,\infty)$ with first moment $m = \int x \,d\pi(x) < \infty$, and we assume $\pi$ is not supported on $\{0\}$, so that $\mathcal{L}_\pi$ is non-trivial. Let $\lambda$ be the leading eigenvalue of $\mathcal{L}_\pi$, with eigenfunction $\theta$ normalized by $\int \theta(x) \,d\pi(x) = 1$. We may write $\pi = p\delta_0 + \pi_{+}$ where $p \ge 0$ and $\pi_{+}$ is a non-trivial positive Borel measure on $(0,\infty)$. Then $\mathcal{L}_{\pi} = \mathcal{L}_{\pi_{+}}$ and their normalized leading eigenfunctions coincide. The normalization conditions are identical because the eigenfunctions take the value $0$ at $0$.

We have $$ \lambda \theta(y) = \int (x \wedge y) \theta(x) d\pi(x) \le y \int \theta(x) \,d\pi(x) = y$$
Therefore \begin{equation}\label{eq: theta slope bound}\theta(y) \le y/\lambda \quad\text{ for all $y$}\,.\end{equation}
The limit $\theta(\infty) = \lim_{x \to\infty} \theta(x)$ can be arbitrarily large, as we saw from the example $\pi = (1-p)\delta_0 + p\delta_{1/p}$, where $\theta(\infty) = 1/p$. But from the proof of Proposition~\ref{P: theta-pi correspondence} we may extract the following explicit bound on $\theta(\infty)$ in terms of the tail of $\pi$.
\begin{lemma} \label{L: explicit theta bound} Let $\pi \in \mathcal{P}_1([0,\infty))\setminus\{\delta_0\}$ and let $\mathcal{L}_\pi$ have leading eigenvalue $\lambda$ and leading eigenfunction $\theta$ normalized by $\int \theta(x)\,d\pi(x) = 1$. Let $0 < C < 1$ and suppose $$\int \mathbf{1}(x > x_0) \,x\, d\pi(x) \le \lambda C\,.$$ Then \begin{equation}\label{eq: explicit theta bound}\theta(\infty) \le \frac{x_0}{\lambda} \,\exp(C/(1-C))\,.\end{equation}
\end{lemma}
\begin{proof} For any $c \in (C,1)$ we have $\int \mathbf{1}(x > x_0) x\,d\pi(x) < \lambda C$ so  \eqref{eq: explicit theta bound in proof} holds. Combining this with \eqref{eq: theta slope bound}  and taking the limit as $c \searrow C$ we obtain~\eqref{eq: explicit theta bound}.
\end{proof}

\begin{lemma}\label{L: lambda lower bound} Let $\pi \in \mathcal{P}_1([0,\infty))$ and let $\lambda$ be the leading eigenvalue of $\mathcal{L}_\pi$. For every $x > 0$ we have $$\lambda \ge x \pi([x,\infty))\,.$$
\end{lemma}
\begin{proof}
We have
\begin{eqnarray*} \lambda \theta(x)  =  \int (x \wedge y) \theta(y)\,d\pi(y)  & \ge & \int \mathbf{1}(y \ge x) (x \wedge y) \theta(y)\,d\pi(y)\\ & \ge & \int \mathbf{1}(y \ge x) x\theta(x) \,d\pi(y) = x \theta(x) \pi([x,\infty))\,. 
\end{eqnarray*}
Since $\theta(x) > 0$ we obtain the desired inequality.
\end{proof}


\subsection{$\lambda$ is a Lipschitz function of $\pi$ with respect to $W_1$}
\begin{lemma}\label{L: f theta is Lip if f is Lip}
 Let $\pi \in \mathcal{P}_1([0,\infty))\setminus\{\delta_0\}$. Suppose $f: [0,z] \to \mathbb{R}$ is $1$-Lipschitz with $f(0) = 0$. Then $x \mapsto f(x)\theta(x)$ is $2\theta(z)$-Lipschitz on $[0,z]$. If $f: [0,\infty) \to \mathbb{R}$ is $1$-Lipschitz with $f(0) = 0$ then $x \mapsto f(x)\theta(x)$ is $2\theta(\infty)$-Lipschitz on $[0,\infty)$.
\end{lemma}
\begin{proof}
For $0 \le x < y \le z$ we have $ 0 \le \theta(y) - \theta(x) \le \frac{(y-x)\theta(y)}{y}$, because $\theta$ is increasing and concave.
Hence
\begin{eqnarray*}|f(y)\theta(y) - f(x) \theta(x)| & \le & |f(y)(\theta(y) - \theta(x))| + |\theta(x)(f(y) - f(x))|\\
& \le & y\,.\,\frac{(y-x)\theta(y)}{y} + \theta(x)|y-x| \\ & = & 2 \theta(z) |y-x|\,.
\end{eqnarray*}
\end{proof}
\begin{lemma}\label{L: useful function is Lip}
For each fixed $x \ge 0$, the function $f_x: y \mapsto \theta(y)(x \wedge y)$ is Lipschitz with constant $2\theta(x)$. 
\end{lemma}
\begin{proof}
On $[x,\infty)$, $f_x(y) = x\theta(y)$ which is $x\theta'_r(x)$-Lipschitz since $\theta$ is concave and increasing. But $x\theta'_r(x) \le \theta(x)$ since $\theta$ is concave and $\theta(0) = 0$. So on $[x,\infty)$, $f_x$ is $\theta(x)$-Lipschitz.

On $[0,x]$, $f_x(y) = y\theta(y)$ and we apply Lemma~\ref{L: f theta is Lip if f is Lip}.
\end{proof}

The next lemma shows that $\lambda$ is a $2$-Lipschitz functional of $\pi$ with respect to the $W_1$ metric. \begin{lemma}\label{L: lambda Lipschitz bound}
 Suppose $\pi$, $\tilde{\pi} \in \mathcal{P}_1([0,\infty)) \setminus\{\delta_0\}$, and let $\theta, \tilde{\theta}$ be the normalized leading eigenfunctions of $\mathcal{L}_{\pi},  \mathcal{L}_{\tilde{\pi}}$ with eigenvalues $\lambda, \tilde{\lambda}$ respectively. 
Then
\begin{equation}\label{eq: lambda Lipschitz bound}\tilde{\lambda} \ge \lambda - 2 W_1(\tilde{\pi},\pi)\,.\end{equation}
\end{lemma}
\begin{proof}
 Since $\theta$ is bounded and continuous it represents a non-negative non-trivial element of $L^2(\tilde{\pi})$.
Using Lemma~\ref{L: useful function is Lip} we have 
\begin{eqnarray*}\mathcal{L}_{\tilde{\pi}} \theta(x) &  = & \int \theta(y) (x \wedge y) \,d\tilde{\pi}(y) \\ & \ge & \int \theta(y) (x \wedge y) \, d\pi(y) - 2\theta(x) W_1(\tilde{\pi},\pi)\\ & = & \left(\lambda - 2W_1\left(\tilde{\pi},\pi\right)\right) \theta(x)\,.\end{eqnarray*}
Therefore
 $$\tilde{\lambda}  = \|\mathcal{L}_{\tilde{\pi}}\| \ge \frac{ \|\mathcal{L}_{\tilde{\pi}}\theta\|_{\tilde{\pi}}}{\|\theta\|_{\tilde{\pi}}} \ge \lambda - 2W_1\left(\tilde{\pi},\pi\right)\,.$$
\end{proof}

\begin{lemma}\label{L: theta locally uniformly bounded}
The mapping $\pi \mapsto \theta(\infty)$ is locally bounded on $\mathcal{P}_1([0,\infty))$.
\end{lemma}
\begin{proof}
Let $\pi \in \mathcal{P}_1\setminus \{\delta_0\}$.  We may choose $r < \infty$ such that $\pi \in \mathcal{P}_{r,\lambda/4}$. Suppose $W_1(\pi,\tilde{\pi}) < \lambda/8$.  Then by Lemma~\ref{L: metric nestedness} with $r' = 2r$ we have $\tilde{\pi} \in \mathcal{P}_{2r, \lambda/2}$.  Let the normalized leading eigenfunction of $\mathcal{L}_{\tilde{\pi}}$ be $\tilde{\theta}$ with leading eigenvalue $\tilde{\lambda}$. By Lemma~\ref{L: lambda Lipschitz bound} we have $\tilde{\lambda} > 3\lambda/4$. 
Since $\tilde{\pi} \in \mathcal{P}_{2r, \lambda/2}$, we have 
$$ \int \mathbf{1}(x \ge 2r) x \,d\tilde{\pi}(x) \le \frac{\lambda}{2} \le \frac{3\tilde{\lambda}}{8} < \frac{\tilde{\lambda}}{2}\,,$$
so applying Lemma~\ref{L: explicit theta bound} with $C=1/2$ gives us
$$ \tilde{\theta}(\infty) \le \frac{2 e r}{\tilde{\lambda}} \le \frac{8 e r}{3\lambda} \,.$$
\end{proof}

\subsection{Initial control of solutions}

Theorem~\ref{T: well-posedness theorem} is concerned with solutions of \eqref{eq: IVP} where in fact $\varphi(t) \le 1$ because of the condition \eqref{eq: phi defined by theta}. For the moment we consider solutions where $\varphi: [0,T] \to [0,1]$ is an arbitrary continuous control function, not necessarily agreeing with $\Phi(\pi_t)$.

\begin{lemma}\label{L: theta bounded for time 1/8}  Let $\pi_t$ be a solution of \eqref{eq: IVP} over a compact time interval $[0,T]$ with initial condition  $\pi_0 \in \mathcal{P}_1([0,\infty)) \setminus \{\delta_0\}$, for any continuous control function $\varphi(t): [0,T] \to [0,1]$. Then $\lambda_t$ is bounded away from $0$ over $t \in [0,T]$ with a positive lower bound that depends only on $\pi_0$, and $\theta_t$ is uniformly bounded over $t \in [0, T]$, with a bound that depends only on $\pi_0$ and $T$. 
\end{lemma}
\begin{proof}
Firstly, we bound $\lambda_t$ away from $0$ for a short time in terms of $\pi_0$. Since $\pi_0$ is not supported on $\{0\}$ there exist $x_1, \epsilon > 0$ such that $\pi_0([x_1, \infty)) > \epsilon$.  
We have $$\varphi(t) \mu_t([0,\infty)) = \varphi(t) \le 1\,.$$ Therefore, applying equation~\eqref{eq: wave test function} to suitable non-negative test functions bounded above by $\mathbf{1}(x > x_1)$, we deduce
$$\pi_t([x_1,\infty)) \ge \pi_t([x_1+t,\infty) \ge \pi_0([x_1,\infty)) - t\,.$$
Thus for $t \le \epsilon/2$ we have $$\pi_t([x_1,\infty)) \ge \epsilon/2\,,$$ and so by Lemma~\ref{L: lambda lower bound} we have $\lambda_t \ge x_1 \epsilon/2$ for $t \le \epsilon/2$.

Secondly, we bound $\lambda_t$ below in a way that does not depend on $\pi_0$. Applying equation~\eqref{eq: wave test function} to non-negative $C_0^1$ test functions supported on $(-\infty,s-t]$ and using $\varphi(t) \le 1$, we find that for $0 \le s \le t$, we have
$\pi_t([0,s)) \le s$ and hence $\pi_t([s,\infty)) \ge 1 - s$. Applying Lemma~\ref{L: lambda lower bound}, we deduce
\[ \lambda_t \ge \sup_{0 \le s \le t} s(1-s) = \begin{cases} t(1-t), & (0 \le t \le 1/2), \\ \frac{1}{4}, & (t \ge 1/2)\,. \end{cases}\]

Combining these two lower bounds for $\lambda_t$ yields $b(\pi_0) > 0$ such that $\lambda_t \ge b(\pi_0)$ for all $t$.

Choose any $c \in (0,1)$. Since $\pi_0$ has finite mean, there exists $x_0 > 0$ such that both $$\int \mathbf{1}(x \ge x_0) x \,d\pi_0(x) \le \frac{b(\pi_0)c}{2}$$ and 
$$\pi_0([x_0,\infty)) \le \frac{b(\pi_0) c}{2T}\,. $$ Then for all $t \le T$ we have 
\begin{multline*} \int \mathbf{1}(x \ge x_0 + t) x \,d\pi_t(x) \le \int \mathbf{1}(x \ge x_0) (x + t)\,d\pi_0 \\ \le \frac{b(\pi_0)}{2} c + \frac{b(\pi_0) c t}{2T} \le b(\pi_0)c \le \lambda_t c\,,\end{multline*} so that by Lemma~\ref{L: explicit theta bound} we have for all $t \in [0,T]$ that
$$ \|\theta_t\|_\infty  = \theta_t(\infty) \le \frac{(x_0+t)}{\lambda_t} \exp(c/(1-c))\le \frac{x_0 + T}{b(\pi_0)} \exp(c/(1-c))\,.$$
\end{proof}
At this point we recall Lemma~\ref{L: speed bound}. Now that we know that both the mean of $\pi_t$ and $\|\theta_t\|_\infty$ are bounded for $t \in [0, T]$, we may deduce the following.
\begin{corollary}\label{Cor: speed limit}
 In the situation of Lemma~\ref{L: theta bounded for time 1/8}, $\pi_t$ moves at bounded speed during the time interval $[0,T]$ with respect to the $W_1$ metric. The bound on the speed depends only on $\pi_0$ and $T$. 
\end{corollary}

\subsection{Locally Lipschitz dependence of $\theta$ on $\pi$}\label{SS: theta is continuous in t}
The goal of this section is to prove the following proposition:
\begin{proposition}\label{P: theta cts wrt pi}
 The normalized leading eigenfunction $\theta$ of $\mathcal{L}_\pi$ depends continuously on $\pi$, as $\pi$ ranges over $\mathcal{P}_1([0,\infty))\setminus \{\delta_0\}$. More precisely,  $\pi \mapsto \theta$ is locally Lipschitz from $W_1$ to $L^\infty$.
\end{proposition}
Remark: taking $\pi_0  = \pi^p = (1-p)\delta_0 + p\delta{1/p}$, the solution of \eqref{eq: IVP} + \eqref{eq: phi defined by theta} initially moves at speed $1+O(p)$ with respect to $W_1$, yet $\frac{d}{dt}\theta_t(1/p)|_{t = 0} = 1/p$. This example shows that the mapping $\pi \mapsto \theta$ is not globally Lipschitz from $W^1$ to $L^\infty$, even if we restrict attention to the subspace of age-critical measures.

\begin{lemma}\label{L: theta over x is Lip wrt pi}
 The map $\pi \mapsto ( x \mapsto \theta(x)/x)$ is locally Lipschitz from $W_1$ to $L^\infty((0,\infty))$, on $\mathcal{P}_1([0,\infty))\setminus\{\delta_0\}$.
\end{lemma}
\begin{proof}
 Let $\pi, \pi_1, \pi_2 \in \mathcal{P}_1([0,\infty))$ and let $m= \int x\,d\pi(x)$.  Let $\theta, \theta_1$ and $\theta_2$ be the normalized principal eigenfunctions of the operators $\mathcal{L}_{\pi}, \mathcal{L}_{\pi_1}$ and $\mathcal{L}_{\pi_2}$, with leading eigenvalues $\lambda$, $\lambda_1$ and $\lambda_2$ respectively.  Suppose also that $\pi_1$ and $\pi_2$ lie in the neighborhood of $\pi$ defined by $W_1(\pi_i,\pi) \le \lambda/4$, so that by Lemma~\ref{L: lambda Lipschitz bound} we have $\lambda/2 \le \lambda_i \le 2\lambda$ for $i=1,2$.
We have to find a constant $C(\pi)$ depending only on $\pi$ such that for all $x \in (0,\infty)$, $$\left|\frac{\theta_1(x)}{x} - \frac{\theta_2(x)}{x}\right| \le C(\pi) W_1(\pi_1,\pi_2)\,.$$
For the remainder of this proof we will use the symbol $\int_0^x$ to denote the integral over the open interval $(0,x)$. For $i=1,2$ and for each $x \ge 0$ we have
\begin{eqnarray} \notag \lambda \theta_i(x)  = \int (x \wedge y) \theta_i(y) \,d\pi_i(y) & = & \int (x - (x-y)_{+}) \theta_i(y)\,d\pi_i(y)\\ & = & x - \int_0^x \theta_i(y)(x-y)\,d\pi_i(y),\label{eq: inhomogeneous Fredholm for theta}\end{eqnarray}
because $\int \theta_i(y)\,d\pi_i(y) = 1$, and the integrand in the final integral vanishes at the endpoints $0$ and $x$.  
Hence
$$ \lambda_1\theta_1(x) - \lambda_2\theta_2(x) \;=\;  \int_0^x \theta_1(y)(x-y) \,d\pi_1(y) - \int_0^x \theta_2(y)(x-y) \,d\pi_2(y) $$
\begin{equation}  = \; \int_0^x (\theta_1(y) - \theta_2(y)) (x-y) \,d\pi_1(y)\; + \; \int_0^x \theta_2(y) (x-y) \,d(\pi_2 - \pi_1)(y)\,.\label{eq: difference of thetas}
\end{equation}

To control the last integral, we use the dual formulation of the $W_1$ distance. For any $x > 0$, we claim the function $y \mapsto \theta_2(y)(x-y)$ is Lipschitz on the interval $[0,x]$, with Lipschitz constant $x/\lambda_2$. We know that $\theta_2$ is $1/\lambda_2$-Lipschitz and $\theta_2(0) = 0$. Therefore for $0 \le z \le y \le x$ we have
\begin{eqnarray*}|\theta_2(y)(x-y) - \theta_2(z)(x-z)| & \le & |(\theta_2(y) - \theta_2(z))(x-y)| + |\theta_2(z)(y-z)| \\ & \le &\frac{1}{\lambda_2}\left( (y-z)(x-y) + z(y-z)\right) \\ & = & \frac{(x-y + z)(y-z)}{\lambda_2} \; \le \; \frac{x}{\lambda_2} \,|y-z|.\end{eqnarray*}
Therefore
\begin{equation}\label{eq: integral bound via dual W1} \left|\int_0^x \theta_2(y) (x-y) \,d(\pi_2 - \pi)(y) \right| \le \frac{x}{\lambda_2}\,W_1(\pi_2,\pi_1)\,.\end{equation}
Substituting the bound~\eqref{eq: integral bound via dual W1} into~\eqref{eq: difference of thetas} and applying the triangle inequality and $(x-y) < x$ yields
$$|\lambda_1\theta_1(x) - \lambda_2\theta_2(x)| \le x \int_0^x | \theta_1(y) - \theta_2(y)|  \,d\pi_1(y) + \frac{x}{\lambda_2} W_1(\pi_2,\pi_1)\,.  $$
Since $\lambda_2 \ge \lambda/2$ we have
\begin{multline*}  \frac{\lambda}{2}\, \left|\theta_1(x)  -  \theta_2(x)\right| \; \le \;  \lambda_2\left|\theta_1(x) -\theta_2(x)\right| \\ \le \left|\lambda_1\theta_1(x) - \lambda_2\theta_2(x)\right| \;+\; \left|\lambda_2-\lambda_1\right|\theta_1(x)\\  \le   x \int_0^x \left|\theta_1(y) - \theta_2(y)\right|  \,d\pi_1(y) \,+\, \frac{x}{\lambda_2} W_1(\pi_2,\pi_1)\,+ \,\frac{2x}{\lambda_1}W_1(\pi_2,\pi_1)\\ \le  x \int_0^x \left|\theta_1(y) - \theta_2(y)\right|  \,d\pi_1(y) \,+\, \frac{2x}{\lambda} W_1(\pi_2,\pi_1)\,+\,\frac{4x}{\lambda}W_1\left(\pi_2,\pi_1\right) \\ \le  x \int_0^x \left|\theta_1(y) - \theta_2(y)\right|  \,d\pi_1(y) \,+\, \frac{6x}{\lambda} W_1\left(\pi_2,\pi_1\right)\,. \phantom{+\, \frac{6x}{\lambda} W_1\left(\pi_2,\pi_1\right)\,.} \end{multline*}

Noting that $\theta_1(0) = \theta_2(0) = 0$, we may rearrange to obtain for every $x \in (0,\infty)$ that
\begin{equation}\label{eq: inequality suitable for Gronwall} \frac{|\theta_1(x) - \theta_2(x)|}{x} \le \frac{12W_1(\pi_2,\pi_1)}{\lambda^2} + \int_0^x \frac{|\theta_1(y) - \theta_2(y)|}{y} \,\frac{2y}{\lambda}\,d\pi_1(y) \end{equation}
We have $\int x \,d\pi_1(x) \le m + \lambda/4$ because the mean is a $1$-Lipschitz functional with respect to $W_1$. Thus the measure $\nu$ defined by $\frac{d\nu}{d\pi_1}(y) = 2y/\lambda$ is finite with total mass $2m_1/\lambda \le 2(m+\lambda/4)/\lambda = \tfrac{1}{2} + 2m/\lambda $, so inequality~\eqref{eq: inequality suitable for Gronwall} allows us to apply Gr\"onwall's inequality, obtaining for all $x > 0$
\begin{eqnarray}
\notag
\left|\frac{\theta_1(x)}{x} - \frac{\theta_2(x)}{x}\right| & \le &  \frac{12}{\lambda^2}W_1(\pi_2,\pi_1) \left(1 + \int_0^x \exp\nu((y,x)) d\nu(y)\right) \\ \notag & \le & \frac{12}{\lambda^2}W_1(\pi_2,\pi_1)\left(1+\left(\frac{1}{2} + \frac{2m}{\lambda}\right) \exp \left(\frac{1}{2} + \frac{2m}{\lambda}\right)\right)\\ & = & C(\pi)\,W_1(\pi_1,\pi_2)\,.\label{eq: Gronwall bound for theta difference}
\end{eqnarray}
\end{proof}

\begin{proof}[Proof of Proposition~\ref{P: theta cts wrt pi}]
Let $\pi \in \mathcal{P}_1$ have mean $m = \int x \,d\pi(x) < \infty$, and let ${\pi}_1, {\pi}_2 \in \mathcal{P}_1$ satisfy $W_1({\pi}_i,\pi) < \lambda/8$ for $i=1,2$. We may choose $r < \infty$ such that $\pi \in \mathcal{P}_{r,\lambda/4}$. By Lemma~\ref{L: metric nestedness} with $r' = 2r$ we have ${\pi}_i \in \mathcal{P}_{2r, \lambda/2}$ for $i=1,2$.  Let the normalized leading eigenfunctions of $\mathcal{L}_\pi$ and 
$\mathcal{L}_{{\pi}_i}$ be $\theta$ and ${\theta}_i$ respectively, with eigenvalues $\lambda, {\lambda}_i$. By Lemma~\ref{L: lambda Lipschitz bound} we have 
$$\frac{3\lambda}{4} \le \lambda_i \le \frac{5\lambda}{4} \quad \text{for $i=1,2$, and }\quad |\lambda_1 - \lambda_2| \le 2W_1(\pi_1,\pi_2)\,.$$
From Lemma~\ref{L: theta locally uniformly bounded} we have for all $x \ge 0$, 
$${\theta}_i(x) \le {\theta}_i(\infty) \le \frac{4 e r}{\lambda} \quad \text{for $i=1,2$}\,.$$

Our goal is to bound $|{\theta}_1(x) - {\theta}_2(x)|$ uniformly in $x$, with the bound depending linearly on $W_1({\pi}_1,{\pi}_2)$, with a constant that depends only on $\pi$. Proceeding as in the previous lemma, we have
\begin{multline} \label{eq: theta difference 1} \frac{3\lambda}{4}\left|{\theta}_1(x) - {\theta}_2(x) \right| \le {\lambda}_2\left|{\theta}_1(x) - {\theta}_2(x) \right|  \le  \left| {\lambda}_1{\theta}_1(x) - {\lambda}_2 {\theta}_2(x) \right| + \left| {\lambda}_2 - {\lambda}_1\right| {\theta}_1(x) \\   \le  \left| {\lambda}_1{\theta}_1(x) - {\lambda}_2{\theta}_2(x) \right| + \frac{8 e r}{\lambda} W_1\left({\pi}_1,{\pi}_2\right)\,.\end{multline}
Now we bound $\left|{\lambda}_1{\theta}_1(x) - {\lambda}_2{\theta}_2(x)\right|$. For any $x \in [0,\infty]$ we have
$$ \lambda_1 \theta_1(x) = \int (x \wedge y) \theta_1(y)\,d\pi_1(y)\,, \quad \lambda_2\theta_2(x) = \int (x \wedge y) \theta_2(y) \,d\pi_2(y)$$
so
\begin{multline}\label{E: difference of theta and theta tilde}\left| \lambda_2\theta_2(x) - \lambda_1\theta_1(x)\right|  \le  \int (x \wedge y) \left| \theta_2(y) - \theta_1(y)\right| \,d\pi_1(x) \\ +\; \left| \int (x \wedge y)\theta_2(y)\,d(\pi_2 - \pi_1)(y)\right|\,. 
\end{multline}
We may bound the second integral on the right-hand side of \eqref{E: difference of theta and theta tilde} using the dual formulation of $W_1$, since by Lemma~\ref{L: useful function is Lip}, the function $y \mapsto (x \wedge y)\theta_2(y)$ is Lipschitz with constant $2\theta_2(x)$, which is at most $2\theta_2(\infty) \le 8 e r/ \lambda$.  Hence
$$\left|\int (x \wedge y)\theta_2(y)\,d(\pi_2 - \pi_1)(y)\right| \;\le\; \frac{8 e r}{\lambda} W_1(\pi_2,\pi_1)\,.$$
Splitting the first integral on the right-hand side of \eqref{E: difference of theta and theta tilde} at $2r$, we rewrite $\int (x \wedge y) \left| \theta_2(y) - \theta_1(y)\right| \,d\pi_1(x)$ as
\begin{multline*}  \int \mathbf{1}(y < 2r) (x \wedge y) \left| \theta_2(y) - \theta_1(y)\right|\,d\pi_1(y) \\ {\;+\; \int \mathbf{1}(y \ge 2r) (x \wedge y) \left| \theta_2(y) - \theta_1(y)\right|\,d\pi_1(y) } \\  \le \; \int \mathbf{1}(y < 2r)y^2 \sup_{y < 2r} \left| \frac{\theta_2(y)}{y} - \frac{\theta_1(y)}{y}\right| \,d\pi_1(y)\;\\ {+\; \| \theta_2 - \theta_1\|_{\infty} \int \mathbf{1}(y \ge 2r) y \,d\pi_1(y) } \\
 \le \; 4r^2\, C(\pi) W_1(\pi_2,\pi_1) + \frac{\lambda}{2}\| \theta_2 - \theta_1\|_{\infty}\,.
\end{multline*}
Adding these bounds we have
\begin{equation}\label{eq: theta difference 2}\left| \lambda_2\theta_2(x) - \lambda_1\theta_1(x)\right|  \;\le\; 4r^2\, C(\pi) W_1(\pi_2,\pi_1) \,+\, \frac{\lambda}{2}\| \theta_2 - \theta_1\|_{\infty}\,+\,8 e r/ \lambda\,.\end{equation}
Substituting the bound~\eqref{eq: theta difference 2} into~\eqref{eq: theta difference 1}, we obtain
$$ \frac{3\lambda}{4}\|\theta_2 - \theta_1\|_\infty \le 4r^2C(\pi)W_1(\pi_2,\pi_1) + \frac{\lambda}{2}\|\theta_2 - \theta_1\|_\infty + \frac{8 e r}{\lambda} W_1(\pi_2,\pi_1)$$ 
Because $\theta_1$ and $\theta_2$ are both bounded, we have $\|\theta_2 - \theta_1\|_\infty < \infty$, so we may rearrange to obtain
$$ \|\theta_2 - \theta_1\|_\infty \le 4\left(4r^2 C(\pi) + \frac{8er}{\lambda}\right)\,W_1(\pi_2,\pi_1)\, = \, C'(\pi)\, W_1(\pi_1,\pi_2)\,.$$

\end{proof}

\subsection{Better control of solutions} 
 We can now control how quickly $\tau_{-t}\pi_t$ decays as $t$ grows.
\begin{lemma}\label{L: limited decay}
In the situation of Lemma~\ref{L: theta bounded for time 1/8} there exists a constant $C > 0 $ depending only on $\pi_0$ and $T$ such that for all $t \in [0,T]$ and all $s$ in the support of $\pi_0$, we have $$\frac{d\tau_{-t}\pi_t}{d\pi_0}(s) \ge e^{-Ct} \,.$$ 
\end{lemma} 
\begin{proof}
 For any non-negative $f \in C^1_0(\mathbb{R})$  we have
\begin{eqnarray*} \int f(s) \,d(\tau_{-t}\pi_t)(s) & = & \int f(s-t)\,d\pi_t(s) \\ & \ge &\int f(s)\,d\pi_0(s) - \int_0^t \varphi(u) \int \theta_u(s) f(s-u)\,d\pi_u(s) \,du \\
& = & \int f(s)\,d\pi_0(s) - \int_0^t \varphi(u) \int \theta_u(s) f(s)\,d(\tau_{-u}\pi_u)(s) \,du \\
& \ge & \int f(s)\,d\pi_0(s) - \sup_{u \in [0,t]}(\varphi(u)\theta_u(\infty))\,\int_0^t  \int f(s)\,d(\tau_{-u}\pi_u)(s) \,du  \end{eqnarray*}
For each $t  \in [0,T]$, write $I(t)= \int f(s)\,d(\tau_{-t}\pi_t)(s)$. Then the above inequality reads
\begin{equation} \label{eq: I decay} I(t) \ge I(0) - \sup_{u \in [0,t]}(\varphi(u)\theta_u(\infty))\,\int_0^t I(u)\,du\,.\end{equation}
From the proof of Lemma~\ref{L: theta bounded for time 1/8}, taking $c=1/2$, and the assumption that $\varphi$ takes values in $[0,1]$, we have $$\sup_{u \in [0,t]}(\varphi(u)\theta_u(\infty)) \le \frac{e(x_0+ T)}{b(\pi_0)}\,,$$ which together with ~\eqref{eq: I decay} implies 
$$I(t) \ge I(0) \exp\left(-\frac{e(x_0+T)}{b(\pi_0)}t\right)\,.$$ Since this holds for every non-negative test function $f \in C_0^1(\mathbb{R})$ we deduce the claimed inequality of measures.
\end{proof}

\begin{lemma}\label{L: phi Lipschitz wrt pi}
Define the total burning rate function $\Phi$ by 
$$ \Phi(\pi) = \left(\int \theta(x)^3 \,d\pi(x) \right)^{-1}\,.$$ Then $\Phi(\pi) \le 1$ for all $\pi \in \mathcal{P}_1([0,\infty))$, and $\Phi$ is locally Lipschitz with respect to the $W_1$ metric on $\mathcal{P}_1([0,\infty))$.
\end{lemma}
\begin{proof}
Jensen's inequality gives $\int \theta^3(x) \,d\pi(x) \ge \left(\int \theta(x) \,d\pi(x)\right)^3 = 1$, so $\Phi(\pi) \le 1$. 

Since $y \mapsto 1/y$ is $1$-Lipschitz on $[1,\infty)$, it suffices to show that $\int \theta^3(x) \,d\pi(x)$ is a locally Lipschitz function of $\pi$. Around any point of $\mathcal{P}_1([0,\infty))$ we may find a neighbourhood $U$  on which the map $\pi \mapsto \theta$ is $C$-Lipschitz from $W_1$ to $L^\infty$, on which $\|\theta\|_\infty$ is bounded by $B$, and on which $\lambda > \epsilon > 0$.
Let $\pi, \tilde{\pi} \in U$ and let $\mathcal{L}_\pi$, $\mathcal{L}_{\tilde{\pi}}$ have normalized leading eigenfunctions $\theta$, $\tilde{\theta}$.
 Then
\begin{multline*} \left| \int \theta(x)^3\,d\pi(x) - \int \tilde{\theta}(x)^3\,d\tilde{\pi}(x) \right| \le \left| \int \theta(x)^3 -  \tilde{\theta}(x)^3\,d\tilde{\pi}(x) \right| \\+\; \left| \int \theta(x)^3\,d(\pi-\tilde{\pi})(x) \right|\,.\end{multline*}
We have $|\theta(x) - \tilde{\theta}(x)| \le C W_1(\pi, \tilde{\pi})$ so the first integral on the RHS is bounded by $3 C B^2 W_1(\pi,\tilde{\pi})$. For the second term, note that $x \mapsto \theta(x)$ is $1/\lambda$-Lipschitz, therefore $1/\epsilon$-Lipschitz. Hence $x \mapsto \theta(x)^3$ is $3B^2/\epsilon$-Lipschitz, and the second integral is bounded by $(3B^2/\epsilon) W_1(\pi,\tilde{\pi})$. 
\end{proof}

\begin{lemma}\label{L: everything Lipschitz on sausage}
Let $\pi_t: t \in [0,T]$ be a solution of \eqref{eq: IVP} for some continuous control function $\varphi: [0,T] \to [0,1]$. Then there exist $C < \infty$ and $\delta > 0$, depending on $\{\pi_t: t \in [0,T]\}$, such that on the set
$$S_\delta:= \{\pi \in \mathcal{P}_1\,:\, W_1(\pi, \{\pi_t: t \in [0,T]\}) < \delta\}\,,$$ the map $\pi \mapsto \theta$  is $C$-Lipschitz from $W_1$ to $L^\infty$, the map $\pi \mapsto \Phi(\pi)$ is $C$-Lipschitz from $W_1$ to $[0,1]$, and for all $\pi \in S_\delta$, we have $\lambda > 1/C$, $\theta(\infty) < C$ and $\int x\,d\pi(x) < C$. 
\end{lemma}
\begin{proof}
Lemma~\ref{L: limited decay} implies that $\pi_t$ never becomes concentrated on $\{0\}$, so $\lambda_t > 0$ and $\theta_t$ is well-defined for all $t \in [0,T]$. 
Corollary~\ref{Cor: speed limit} says $t \mapsto \pi_t$ is  Lipschitz with respect to $W_1$ over $[0,T]$, in particular continuous, so the track $\{\pi_t: t \in [0,T]\}$ is compact with respect to the $W_1$-topology. By Lemma~\ref{L: lambda Lipschitz bound}, $\lambda$ is continuous and therefore attains its minimum on this track; this minimum is positive, so by taking $\delta$ small enough we ensure $\lambda$ is uniformly bounded away from $0$ on $S_\delta$. The mean functional $\int x \,d\pi(x)$ is $1$-Lipschitz with respect to $W_1$, so it is bounded on $S_\delta$ because $S_\delta$ is a bounded set.  Proposition~\ref{P: theta cts wrt pi} and Lemma~\ref{L: phi Lipschitz wrt pi} show that around each $\pi_t$ we may find an open $W_1$-ball on which the remaining conclusions apply. Passing to a finite subcover of the track shows that we may take $\delta$ small enough and $C$ large enough that all of the conclusions apply on $S_\delta$.
\end{proof}
 
\begin{lemma}\label{L: summary}
Let $\pi_t: t \in [0,T]$ be a solution of \eqref{eq: IVP} for some continuous control function $\varphi: [0,T] \to [0,1]$. Then  
\begin{enumerate}
\item $\pi_t \in \mathcal{P}_1([0,\infty))$ for all $t \in [0,T]$.
\item There exists $\epsilon > 0$ depending only on $\pi_0$ such that $W_1(\pi_t, \delta_0) > \epsilon$ for all $t$. 
\item $t \mapsto \pi_t$ is Lipschitz  with respect to $W_1$, \item $t \mapsto \lambda_t$ is Lipschitz ,\item $t \mapsto \theta_t$ is Lipschitz with respect to $L^\infty$, and
\item $t \mapsto \Phi(\pi_t)$ is Lipschitz.
\end{enumerate}
\end{lemma}
\begin{proof}
Lemma~\ref{L: mean age stays finite} says that solutions of \eqref{eq: IVP} are uniformly integrable, so they stay in $\mathcal{P}_1([0,\infty))$. 
 We have $\lambda_t \le \int x \,d\pi_t(x) = W_1(\pi_t,\delta_0)$ and $\lambda_t$ is bounded away from $0$.
 Corollary~\ref{Cor: speed limit} shows that $\pi_t$ moves with bounded speed over $[0,T]$. Combining this with Lemma~\ref{L: everything Lipschitz on sausage} gives the remaining claims.
\end{proof}

\subsection{Proof of well-posedness of ~\eqref{eq: IVP} + \eqref{eq: phi defined by theta}}

The existence of a solution of  \eqref{eq: IVP} + \eqref{eq: phi defined by theta} and $\lambda_t \equiv 1$ over the time interval $[0,T]$ for an arbitrary $\pi \in \mathcal{P}_1\setminus\{\delta_0\}$ is established by Theorem~\ref{thm_convergence}. So to establish the remaining claims of Theorem~\ref{T: well-posedness theorem} we must show uniqueness of solutions and locally Lipschitz dependence of the solution on the initial condition. This is done in Proposition~\ref{P: well-posed 1} below.

In the following, $\pi_t$, $\tilde{\pi}_t$ and $\hat{\pi}_t$ are any three solutions of \eqref{eq: IVP} + \eqref{eq: phi defined by theta} over the time interval $[0,T]$, with initial values in $\mathcal{P}_1([0,\infty)) \setminus \{\delta_0\}$. The operators $\mathcal{L}_{\pi_t}$, $\mathcal{L}_{\tilde{\pi}_t}$ and $\mathcal{L}_{\hat{\pi}_t}$ have leading eigenfunctions $\theta_t$, $\tilde{\theta}_t$, $\hat{\theta}_t$, respectively, with eigenvalues $\lambda_t$, $\tilde{\lambda}_t$, and $\hat{\lambda}_t$. They define control functions $\varphi(t)$, $\tilde{\varphi}(t)$ and $\hat{\varphi}(t)$ satisfying $\varphi(u) = \Phi(\pi_u)$,  $\hat{\varphi}(u) = \Phi(\hat{\pi}_u)$ and $\tilde{\varphi}(u) = \Phi(\tilde{\pi}_u)$. Let $\delta > 0$ and $0 < C < \infty$ be the constants provided by Lemma~\ref{L: everything Lipschitz on sausage} applied to the solution $\pi_t$. 

The following proposition says that the map $\pi_0 \mapsto \left(\pi_t: t \in [0,T]\right)$ obtained by solving~ \eqref{eq: IVP} and~\eqref{eq: phi defined by theta} is well-defined and locally Lipschitz with respect to $W_1$. 
\begin{proposition}\label{P: well-posed 1} There exist $C_1 < \infty $ and $\gamma > 0$, depending on the solution $\pi_t: t \in [0,T]$, such that
\begin{enumerate} \item If $W_1\left(\pi_0, \tilde{\pi}_0\right) < \gamma$ then  $W_1\left(\hat{\pi}_t, \tilde{\pi_t}\right) < \delta$ for all $t \in [0,T]$. \item If  $W_1\left(\pi_0, \tilde{\pi}_0\right) < \gamma$ and $W_1\left(\pi_0, \hat{\pi}_0\right) < \gamma$ then for all $t \in [0,T]$
\begin{equation}\label{eq: well-posed 1} W_1\left(\hat{\pi}_t,\tilde{\pi}_t\right) \le \,e^{C_1 t} \,W_1\left(\hat{\pi}_0,\tilde{\pi}_0\right)\,.\end{equation} \end{enumerate}
\end{proposition}
\begin{proof}
Consider any test function $f \in \mathrm{Lip}^1(\mathbb{R}) \cap C_0^1(\mathbb{R})$. Applying equation~\eqref{eq: travelling wave integral form} to the solutions $\hat{\pi}_t$ and $\tilde{\pi}_t$, we have
\begin{multline}\label{eq: test difference} \int f(s-t) \,d(\hat{\pi}_t - \tilde{\pi}_t)(s)  - \int f(s) \,d(\hat{\pi}_0 - \tilde{\pi}_0)(s) \;=\; \\ \int_0^t \left(\int f(s-u)\hat{\varphi}(u)\hat{\theta}_u(s)\,d\hat{\pi}_u(s) - \int f(s-u) \tilde{\varphi}(u)\tilde{\theta}_u(s)\,d\tilde{\pi}_u(s)\right)\,du \\ + \int_0^t (\hat{\varphi}(u)-\tilde{\varphi}(u))f(-u)\,du\,.\end{multline}
Let $$ T_1 = \sup\{t \in [0,T] \,|\, \hat{\pi}_u, \tilde{\pi}_u \in S_\delta \; \text{ for all $u \in [0, t]$}\}\,.$$ 
Then for $u < T_1$ we have $$|\Phi(\hat{\pi}_u) - \Phi(\tilde{\pi}_u)| \le C W_1\left(\hat{\pi}_u, \tilde{\pi}_u\right)\,,$$ so for $t \le T_1$ we have
$$\left|\int_0^t (\hat{\varphi}(u)-\tilde{\varphi}(u))f(-u)\,du\right| \le CT \int_0^t W_1\left(\hat{\pi}_u,\tilde{\pi}_u\right)\,du\,.$$
We bound the first integral on the right-hand side of~\eqref{eq: test difference} by
\begin{multline} \int_0^t \hat{\varphi}(u) \left|\int f(s-u)\hat{\theta}_u(s)\,d\!\left(\hat{\pi}_u-\tilde{\pi}_u\right)\!(s)\right| \,du\; \\+\; \int_0^t \int \left| f(s-u) \left(\hat{\varphi}(u)\hat{\theta}_u(s) - \tilde{\varphi}(u)\tilde{\theta}_u(s)\right)\right|\,d\tilde{\pi}_u(s)\,du\,.\label{eq: upper bound for first term in test difference}\end{multline}
The function $s \mapsto f(s-u) \hat{\theta}_u(s) = f(-u)\hat{\theta}_u(s) + (f(s-u) - f(-u))\hat{\theta}_u(s)$ is Lipschitz on $[0,\infty)$ with constant $u/\hat{\lambda} + 2\hat{\theta}_u(\infty)$. This constant is at most $(T+2)C$. Hence the first integral in~\eqref{eq: upper bound for first term in test difference} is at most 
$$ (T+2)C\,\int_0^t W_1\left(\hat{\pi}_u, \tilde{\pi}_u\right)\,du\,.$$ 
Next, we bound the inner integral in the second term of~\eqref{eq: upper bound for first term in test difference}.
\begin{eqnarray*} \left|\hat{\varphi}(u)\hat{\theta}_u(s) - \tilde{\varphi}(u)\tilde{\theta}_u(s)\right| & \le & \hat{\varphi}(u)\|\hat{\theta}_u - \tilde{\theta}_u\|_\infty + \tilde{\theta}(\infty)|\hat{\varphi}(u) - \tilde{\varphi}(u)| \\ & \le & C W_1\left(\hat{\pi}_u, \tilde{\pi}_u\right) + C^2 W_1\left(\hat{\pi}_u, \tilde{\pi}_u\right) \\ & = & (C+C^2)\, W_1\left(\hat{\pi}_u, \tilde{\pi}_u\right)\,.\end{eqnarray*}
For $0 \le u < T_1$ we have
$$ \int |f(s-u)|\,d\tilde{\pi}_u(s) \le \int (u+s)\,d\tilde{\pi}_u(s) \le u + u + \int s\,d\tilde{\pi}_0(s) \le 2u + C \le 2T+C\,. $$
Hence for $t \le T_1$ we have
\begin{multline*}\int_0^t \int \left| f(s-u) \left(\hat{\varphi}(u)\hat{\theta}_u(s) - \tilde{\varphi}(u)\tilde{\theta}_u(s)\right)\right|\,d\tilde{\pi}_u(s)\,du\, \\ \le C(1+C)(2T+C)\int_0^t W_1\left(\hat{\pi}_u, \tilde{\pi}_u\right)\, du\,.  \end{multline*}
Summing these bounds we obtain
\begin{equation}\label{eq: test bound} \left|\int f(s-t)\,d\!\left(\hat{\pi}_t - \tilde{\pi}_t\right)\!(s)\right|  \le  \left| \int f(s)\,d\!\left(\hat{\pi}_0 - \tilde{\pi}_0\right)\!(s)\right| + C_1 \int_0^t W_1\left(\hat{\pi}_u, \tilde{\pi}_u\right)\,du\,.
 \end{equation}
where $C_1 = C(2+4T+C+2TC+C^2)$.  Now take the supremum of both sides of~\eqref{eq: test bound} over $f \in \mathrm{Lip}^1(\mathbb{R}) \cap C^1_0(\mathbb{R})$ and apply the duality criterion of Lemma~\ref{L: restricted duality} to obtain for all $t \in [0,T_1]$
\begin{equation}\label{eq: Gronwallable bound} W_1\left(\hat{\pi}_t, \tilde{\pi}_t\right) \le W_1\left(\hat{\pi}_0, \tilde{\pi}_0\right) + C_1 \int_0^t W_1\left(\hat{\pi}_u, \tilde{\pi}_u\right)\,du\,.
 \end{equation}
Apply Gr\"onwall's inequality to~\eqref{eq: Gronwallable bound} to obtain~\eqref{eq: well-posed 1} for all $t \in [0,T_1]$. Applying the same reasoning to the pair $\pi_t, \tilde{\pi_t}$ and the pair $\pi_t, \hat{\pi}_t$ we also obtain
$$ W_1(\hat{\pi}_t, \pi_t) \le e^{C_1 t} W_1(\hat{\pi}_0, \pi_0)$$ and
$$ W_1(\tilde{\pi}_t, \pi_t) \le e^{C_1 t} W_1(\tilde{\pi}_0, \pi_0)$$
for all $t \in [0,T_1]$.

Define $\gamma = \delta \exp(-C_1 T)$. Then we claim that $T_1 = T$.  Indeed, suppose for a contradiction that $T_1 < T$. Then $$ W_1(\hat{\pi}_{T_1}, \pi_{T_1}) \le e^{C_1 T_1} W_1(\hat{\pi}_0, \pi_0) \le \delta e^{-C_1(T-T_1)} < \delta$$ and similarly $W_1(\tilde{\pi}_{T_1},\pi_{T_1}) \le \delta e^{-C_1(T-T_1)} < \delta$. Since solutions move continuously, both solutions $\tilde{\pi}_t$ and $\hat{\pi}_t$ stay in $S_\delta$ for at least a small interval beyond time $T_1$, contrary to the definition of $T_1$.
\end{proof}

\end{document}